\documentclass[psamsfonts]{amsart}

\usepackage{amssymb,amsfonts}
\usepackage{tikz-cd} 
\usepackage[all,arc]{xy}
\usepackage{enumerate}
\usepackage{mathrsfs}
\usepackage{hyperref}
\hypersetup{
pdftitle={Generalised Temperley-Lieb algebras of type $G(r,1,n)$},
pdfsubject={Mathematics},
pdfauthor={Gus Lehrer AND Mengfan Lyu},
pdfkeywords={Temperley-Lieb algebra, Hecke algebras,cellular basis, decomposition numbers}}
\usepackage{bm}
\usepackage{pgfplots}
\usepackage{cite}

\newtheorem{thm}{Theorem}[section]
\newtheorem{cor}[thm]{Corollary}

\newtheorem{lem}[thm]{Lemma}

\theoremstyle{definition}
\newtheorem{defn}[thm]{Definition}
\newtheorem{exmp}[thm]{Example}

\theoremstyle{remark}

\makeatletter
\let\c@equation\c@thm
\makeatother
\numberwithin{equation}{section}

\title{Generalised Temperley-Lieb algebras of type $G(r,1,n)$}

\author{Gus Lehrer and Mengfan Lyu}
\address{School of Mathematics and Statistics,
University of Sydney, N.S.W. 2006, Australia}
\email{gustav.lehrer@sydney.edu.au,fans4739@gmail.com }

\begin{document}

\begin{abstract}
In this paper, we define a quotient of the cyclotomic Hecke algebra of type $G(r,1,n)$ as a generalisation of the Temperley-Lieb algebras of type $A$ and $B$. We establish a graded cellular structure for the generalised Temperley-Lieb algebra and, using the technology of $KLR$ algebras, determine the corresponding decomposition matrix.
\end{abstract}

\maketitle
\tableofcontents

\section{Introduction}
 
\subsection{Temperley-Lieb algebras of types $A_{n-1}$ and $B_n$}
The Temperley-Lieb algebra was originally introduced in \cite{TL1971} in connection with transition matrices in statistical mechanics. Since then, it has been shown to have links with many diverse areas of mathematics, such as operator algebras, quantum groups, categorification and representation theory. It can be defined in several different ways,including as a quotient of the Hecke algebra and as an associative diagram algebra.
Steinberg pointed out in 1982 during a lecture at UCLA that the Temperley-Lieb algebra $TL_n(q)$ is a quotient of the Hecke algebra $H_n(q)$ of type $A_{n-1}$ by the ideal generated by some idempotents corresponding to two non-commuting generators. It plays a pivotal role in the polynomial invariant of knots and links discovered by Jones in \cite{Jones1985}.
In 1990, Kauffman found a diagrammatic presentation for the Temperley-Lieb algebra in \cite{Kauffman1990AnIO}. It is defined as an associative diagram algebra modulo planar isotopy for classical unoriented knots and links.
In \cite{Bernstein99acategorification}, Bernstein, Frenkel and Khovanov "categorify" the Temperley-Lieb algebra, which is the commutant of the quantum enveloping algebra $U_q(\mathfrak{sl}_2)$ in its action on the $n^{th}$ tensor power $V_1^{\bigotimes n}$ of the natural two dimensional representation $V_1$. More precisely, they realise the Temperley-Lieb algebra via the projective and Zuckerman functors in the usual category $\mathcal{O}$ for $\mathfrak{gl}_n$.
Moreover, the Temperley-Lieb algebra also plays an important role in statistical physics models, and readers can find more references in \cite{Batchelor_1991},\cite{SOS} and  \cite{abramsky2009temperleylieb}.

The ``blob algebra'' $B_n(q,Q)$ was introduced in \cite{Martin_1994} by Martin and Saleur in connection with statistical mechanics. It has a similar diagrammatic representation to the original Temperley-Lieb algebra. Algebraically, it can also be identified with a quotient of $HB_n(q,Q)$ the two-parameter Hecke algebra of type $B_n$. For these reasons, the blob algebra is considered as a type-$B_n$ analogue of $TL_n(q)$ and is also denoted by $TLB_n(q,Q)$. 
In analogy to the original Temperley-Lieb algebra, $TLB_n(q,Q)$ also fits into a category. In \cite{IoharaLehrerZhang2021}, Iohara, Lehrer and Zhang construct a family of equivalences between the Temperley-Lieb category of type B and a certain category of infinite dimensional representations of $U_q(\mathfrak{sl}_2)$ as a generalization of the Temperley-Lieb category $TL(q)$.

Cellularity is a property of associative algebras which permits their deformation and enables the computation of their decomposition numbers. The above algebras are shown to be cellular in \cite{Lehrer1996}. Using this result, the standard modules and decomposition numbers are well-known, see \cite{MARTIN2000957},\cite{CoxGrahamMartin2003} and \cite{GRAHAM2003479}.

\subsection{Objectives of this paper}
In this paper, we define a generalised Temperley-Lieb algebra $TL_{r,1,n}$ in Definition \ref{TL} corresponding to the imprimitive unitary reflection group $G(r,1,n)$, which can be regarded as a generalisation of the Coxeter groups of type $A_{n-1}$ and $B_n$. This generalisation strengthens the connection between Temperley-Lieb algebras and cyclotomic Hecke algebras and brings potential innovations in knot theory, quantum enveloping algebras and statistical physics models. We also construct a cellular structure of our generalised Temperley-Lieb algebra $TL_{r,1,n}$ in Theorem \ref{cb33}, which provides clues for a potential diagrammatic realisation. Using this cellular basis, we give a condition equivalent to the semisimplicity of $TL_{r,1,n}$ in Theorem \ref{Th35}. Moreover, we determine the associative decomposition matrix in Theorem \ref{decomtlr1n}.

Our generalisation is inspired by the construction of generalised blob algebras by Martin and Woodcock. They show in \cite{martin_woodcock_2003} that $TLB_n(q,Q)$ can be obtained as the quotient of $HB_n(q,Q)$ by the idempotents corresponding to the irreducible modules associated with the bipartitions $((2),\emptyset)$ and $(\emptyset,(2))$ and generalised this construction to all cyclotomic Hecke algebras $H(r,1,n)$.

Following this method, we define in Definition \ref{TL} the generalised Temperley-Lieb algebra $TL_{r,1,n}$ as a quotient of the corresponding cyclotomic Hecke algebra $H(r,1,n)$ by the two-sided ideal generated by certain idempotents. The Temperley-Lieb algebras of both type $A_{n-1}$ and $B_n$ are special cases of $TL_{r,1,n}$. Therefore, we refer to this quotient of the cyclotomic Hecke algebra $H(r,1,n)$ as the generalised Temperley-Lieb algebra of type $G(r,1,n)$.

\subsection{The content of this work}
As a direct consequence of the definition, our generalised Temperley-Lieb algebra $TL_{r,1,n}$ is a quotient of the generalised blob algebra $B_{r,n}$. We use this property to construct a graded cellular structure on it.


Lobos and Ryom-Hansen \cite{LOBOSMATURANA2020106277} give a graded cellular basis associated with one-column multipartitions of $n$ for the generalised blob algebras $B_{r,n}$. In \cite{bowman2021graded}, Bowman gives the decomposition numbers for the standard modules of the generalised blob algebras $B_{r,n}$. In Lemma \ref{triquo}, we show that the ideal of $B_{r,n}$ corresponding to its quotient $TL_{r,1,n}$ is exactly spanned by the elements corresponding to a downward closed subset of the poset in the cellular basis given by Lobos and Ryom-Hansen. Therefore, a graded cellular structure of $TL_{r,1,n}$ can be obtained by a truncation of that of $B_{r,n}$, which is described in Theorem \ref{cb33}. The graded cellular basis of $TL_{r,1,n}$ is associated with one-column $r$-partitions of $n$ which consist of at most two non-empty components.

Using this graded cellular basis, we investigate the representations of the generalised Temperley-Lieb algebra $TL_{r,1,n}$. Following the calculation of the dimensions of the cell modules, we concentrate on the bilinear form $\phi_{\lambda}(,)$ on each cell module. For a one-column multipartition $\lambda$, let $W(\lambda)$ be the corresponding cell module of $TL_{r,1,n}$. Define a bilinear form $\phi_{\lambda}(,)$ on $ W(\lambda ) $ in the usual way.According to Theorem 3.4 in \cite{Lehrer1996}, $W(\lambda )$ is simple if and only if $rad(\phi_{\lambda})=0$ . Using this method, we obtain a criterion for the semisimplicity of $TL_{r,1,n}$ in Theorem \ref{Th35} and interpret it as a restriction on the parameters in Corollary \ref{coss}. 

As the main result in this paper, we determine the decomposition numbers for the cell modules using this bilinear form. We first give a method to find out the radical of the bilinear form $\phi_{\lambda}(,)$ in section \ref{6.2}.  By showing that this method is not sensitive to the parameter $r$, we demonstrate that the dimension of a simple $TL_{r,1,n}$-module equals that of a corresponding $TLB_n(q,Q)$-module with the parameters chosen properly.
By showing that the cell and simple modules of $TL_{r,1,n}$ can be realised as those of a corresponding Temperley-Lieb algebra of type $B_n$, we transform the decomposition matrix of $TL_{r,1,n}$ into a combination of those of $TLB_n(q,Q)$ in Lemma \ref{lm52}.

The non-graded decomposition numbers of the Temperley-Lieb algebras of type $B_n$ are given by Martin and Woodcock in \cite{martin_woodcock_2003} and Graham and Lehrer in \cite{GRAHAM2003479} and the graded ones are given by Plaza in \cite{PLAZA2013182}. Following their results, we give the decomposition numbers of $TL_{r,1,n}$ in Theorem \ref{decomtlr1n}.

\section{Multipartitions, Cyclotomic Hecke algebras and KLR algebras} \label{nt}
In this section, we recall some combinatorial concepts such as multipartitions and tableaux in the context of the representation theory of the cyclotomic Hecke algebra as developed by Arike and Koike in \cite{ARIKI1994216}. We will use this language to define our generalized Temperley-Lieb algebra $TL_{r,1,n}$ as a quotient of the cyclotomic Hecke algebra $H(r,1,n)$.
We also recall the cyclotomic KLR algebra $\mathcal{R}_n^{\Lambda}$ of type A, which is isomorphic to $H(r,1,n)$ according to Theorem 1.1 in \cite{Brundan_2009}. We will construct a graded cellular basis of $TL_{r,1,n}$ using the KLR generators in section \ref{gcs}.

\subsection{Multipartitions and their Young tableaux}

A partition of $n$ is a sequence $\sigma = (\sigma_1 \geq \sigma_2 \geq \dots)$ of non-negative integers $\sigma_i$ such that $|\sigma|=\sum_{i\geq 1}\sigma_i=n$. Denote $\sigma = (\sigma_1 , \sigma_2 , \dots, \sigma_k)$ if $\sigma_i=0$ for $i>k$ and $\sigma_k>0$. A multipartition (more specifically, an $r$-partition) of $n$ is an ordered $r$-tuple $\lambda=(\lambda^{(1)},\dots,\lambda^{(r)})$ of partitions with $|\lambda^{(1)}|+\dots+|\lambda^{(r)}|=n$;  if the $i^{th}$ component of $\lambda$ is empty, $|\lambda^{(i)}|=0$. 
Denote by $\mathfrak{P}_n^{(r)}$ the set of $r$-partitions of $n$.

The Young diagram of a multipartition $\lambda\in \mathfrak{P}_n^{(r)}$ is the set of 3-tuples:
\begin{equation}
	[\lambda]:=\{(a,b,l)\;|\; a>0, 1\leq b\leq \lambda_a^{(l)},1\leq l\leq r\}\subset \mathbb{Z}^3.
\end{equation}
For fixed $l$, the subset $[\lambda^{(l)}]:=\{(a,b,l)|a>0, 1\leq b\leq \lambda_a^{(l)}\}$ is the $l^{th}$ component of this Young diagram of the multipartition $\lambda$. It can be regarded as a Young diagram of the partition $\lambda^{(l)}$. If $|\lambda^{(l)}|=0$, we call it an empty component of the multipartition $\lambda$. As there is a unique Young diagram $[\lambda]$ for each multipartition $\lambda\in \mathfrak{P}_n^{(r)}$ and vice versa, we do not distinguish $[\lambda]$ and $\lambda$ in the following sections of this thesis.
For each 3-tuple $(a,b,l)\in [\lambda]$, we call it the node in the $a^{th}$ row and $b^{th}$ column of the $l^{th}$ component. A Young tableau is obtained by filling in the nodes with numbers $1,2,\dots,n$.
More precisely, a Young tableau of shape $\lambda$ is a bijective map:
\begin{equation}
	t:\{1,2,3,\dots,n\}\to [\lambda].
\end{equation}

  For example, let $n=9$, $r=3$ and $\lambda=((4,3),0,(1,1))$. Figure-$\ref{yd}$ shows a Young tableau of shape $\lambda$.

\begin{figure}[h]
	\begin{center}
		\begin{tikzpicture}[scale=1]
			\draw (1,0)--(4,0);\draw (1,1)--(5,1);\draw (1,2)--(5,2);\draw (5,1)--(5,2);
			\foreach \x in {1,2,3,4}
			\draw (\x,2)--(\x,0);
			\draw (7,0)--(8,0); \draw (7,1)--(8,1);\draw (7,2)--(8,2);\draw (7,0)--(7,2);\draw (8,0)--(8,2);
			
			\draw node at (1.5,1.5){1};\draw node at (2.5,1.5){2};\draw node at (3.5,1.5){5};\draw node at (4.5,1.5){8};
			\draw node at (1.5,0.5){3};\draw node at (2.5,0.5){4};\draw node at (3.5,0.5){9};
			\draw node at (7.5,0.5){7};\draw node at (7.5,1.5){6};
			\draw node at (6,1){$\emptyset$};
			\draw node at (0.5,2.5){$l$=};\draw node at (3,2.5){1};\draw node at (6,2.5){2};\draw node at (7.5,2.5){3};
		\end{tikzpicture}
	\end{center}
	\caption{A standard Young tableau of shape $\lambda$ } \label{yd}
\end{figure}

A standard Young tableau is a Young tableau in which the numbers increase along each row and column in every component. Denote by $Std(\lambda)$ the set of standard tableaux of shape $\lambda$. The tableau in Figure-$\ref{yd}$ is a standard Young tableau of shape $((4,3),\emptyset,(1,1))$.
\subsection{One-column multipartitions and the partial order} \label{1cmltp}
 Next we introduce the definition of one-column multipartitions which are the main tool we use to describe the representations of our generalised Temperley-Lieb algebras. A multipartition $\lambda\in \mathfrak{P}_n^{(r)}$ is called a one-column multipartition if each non-empty component consists of only one column. In other words, $\lambda$ is called a one-column multipartition if $\lambda_a^{(l)}\leq 1$ for all $a>0, 1\leq l\leq r$. Figure $\ref{oct}$ gives an example of the one-column multipartition and  one of its standard tableaux. Denote by $\mathfrak{B}_n^{(r)}$ the set of one-column $r$-partitions of $n$. If $\gamma=(i,1,k)$ and $\gamma'=(i',1,k')$ are two nodes of a one-column multipartition $\lambda$, we say $\gamma <\gamma'$ if either $i<i'$ or $i=i'$ and $k<k'$. We write $\gamma \leq \gamma'$ if $\gamma<\gamma'$ or $\gamma=\gamma'$. For example, let $t$ be the tableau in Figure \ref{oct}, then we have $t(1)<t(2)<t(4)<t(3)<t(5)<t(6)<t(7)<t(8)<t(9)$. In fact, $\leq$ is a total order on the set of nodes $N=\{(a,1,l)|a>0, 1\leq l\leq r\}$.
 
 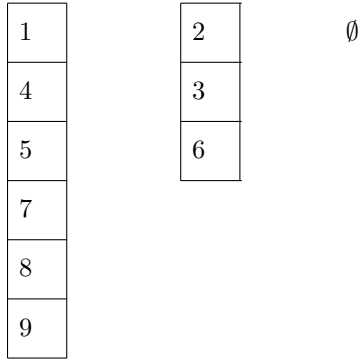
\begin{figure}[h]
	\begin{center}
 \begin{tikzpicture}[x=0.75pt,y=0.75pt,yscale=-1,xscale=1]

\draw  [draw opacity=0] (480,43) -- (509.94,43) -- (509.94,222.16) -- (480,222.16) -- cycle ; \draw   (480,43) -- (480,222.16)(509.86,43) -- (509.86,222.16) ; \draw   (480,43) -- (509.94,43)(480,72.86) -- (509.94,72.86)(480,102.71) -- (509.94,102.71)(480,132.57) -- (509.94,132.57)(480,162.42) -- (509.94,162.42)(480,192.28) -- (509.94,192.28)(480,222.13) -- (509.94,222.13) ; \draw    ;
\draw  [draw opacity=0] (567.34,43) -- (598.15,43) -- (598.15,132.67) -- (567.34,132.67) -- cycle ; \draw   (567.34,43) -- (567.34,132.67)(597.19,43) -- (597.19,132.67) ; \draw   (567.34,43) -- (598.15,43)(567.34,72.86) -- (598.15,72.86)(567.34,102.71) -- (598.15,102.71)(567.34,132.57) -- (598.15,132.57) ; \draw    ;

\draw (484.58,51.67) node [anchor=north west][inner sep=0.75pt]   [align=left] {1};
\draw (484.58,81.52) node [anchor=north west][inner sep=0.75pt]   [align=left] {4};
\draw (484.58,111.38) node [anchor=north west][inner sep=0.75pt]   [align=left] {5};
\draw (484.58,141.23) node [anchor=north west][inner sep=0.75pt]   [align=left] {7};
\draw (484.58,171.09) node [anchor=north west][inner sep=0.75pt]   [align=left] {8};
\draw (484.58,200.95) node [anchor=north west][inner sep=0.75pt]   [align=left] {9};
\draw (571.92,51.67) node [anchor=north west][inner sep=0.75pt]   [align=left] {2};
\draw (571.92,81.52) node [anchor=north west][inner sep=0.75pt]   [align=left] {3};
\draw (571.92,111.38) node [anchor=north west][inner sep=0.75pt]   [align=left] {6};
\draw (649.43,50.09) node [anchor=north west][inner sep=0.75pt]    {$\emptyset $};

\end{tikzpicture}
\end{center}
	\caption{A standard tableau of shape $\lambda=((1^6),(1^3),0)$ } \label{oct}
\end{figure}
 
For $\lambda,\mu\in \mathfrak{B}_n^{(r)}$, we say $\lambda\unlhd \mu$ if for each $\gamma_0\in \mathbb{N}\times \{1\}\times\{1,2\dots,r\}$ we have
\begin{equation}\label{podf}
	|\{\gamma\in[\lambda]:\gamma \leq \gamma_0\}|\geq 	|\{\gamma\in[\mu ]:\gamma \leq \gamma_0\}|.
\end{equation}

Thus, roughly speaking, a multipartition is smaller if its diagram has more `small' nodes with respect to the total order $\leq$. The relation $\unlhd$ defines a partial order on $\mathfrak{B}_n^{(r)}$. The cellular structure of our generalised Temperley-Lieb algebra $TL_{r,1,n}$ will involve this poset. As an example, the diagram in Figure-\ref{fig-po} shows the partial order on $\mathfrak{B}_3^{(3)}$.

	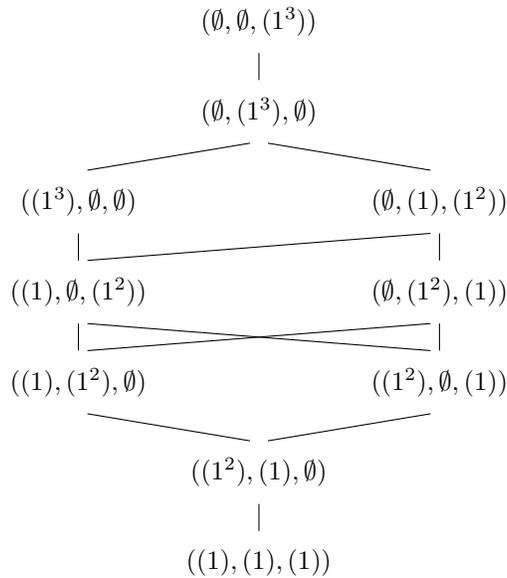
\begin{figure}[ht]
	\begin{center}
		\begin{tikzpicture}[scale=1.2]

			\draw (4,0.6)--(4,0.9);
			\draw (3.9,1.6)--(2.1,1.9);
			\draw (4.1,1.6)--(5.9,1.9);
			\draw (2,2.6)--(2,2.9);
			\draw (6,2.6)--(6,2.9);
			\draw (2.1,2.6)--(5.9,2.9);
			\draw (5.9,2.6)--(2.1,2.9);
			\draw (2.1,3.6)--(5.9,3.9);
			\draw (2,3.6)--(2,3.9);
            \draw (6,3.6)--(6,3.9);
			\draw (3.9,4.9)--(2.1,4.6);
            \draw (4.1,4.9)--(5.9,4.6);
			\draw (4,5.6)--(4,5.9);

			\draw node at (4,0.25){$((1),(1),(1))$};
			\draw node at (4,1.25){$((1^2),(1),\emptyset)$};
			\draw node at (2,2.25){$((1),(1^2),\emptyset)$};
			\draw node at (6,2.25){$((1^2),\emptyset,(1))$};			
			\draw node at (2,3.25){$((1),\emptyset,(1^2))$};			
			\draw node at (6,3.25){$(\emptyset,(1^2),(1))$};
			\draw node at (2,4.25){$((1^3),\emptyset,\emptyset)$};			
			\draw node at (6,4.25){$(\emptyset,(1),(1^2))$};
			\draw node at (4,5.25){$(\emptyset,(1^3),\emptyset)$};			
			\draw node at (4,6.25){$(\emptyset,\emptyset,(1^3))$};			
		\end{tikzpicture}
	\end{center}
	\caption{Partial order on $\mathfrak{B}_3^{(3)}$}\label{fig-po}
\end{figure}

For $1\leq k\leq r$, denote by $\mathfrak{B}_n^{(r)}(k)$ the subset of $\mathfrak{B}_n^{(r)}$ consisting of the multipartitions containing exactly $k$ non-empty components. We next discuss some consequences of the definition of the partial order.
\begin{lem}
    
$\label{dw}$
	Let $\lambda\in \mathfrak{B}_n^{(r)}(k)$ and $\mu \in \mathfrak{B}_n^{(r)}(l)$ be two $r$-partitions of $n$. If $\lambda\unlhd \mu$, then $k\geq l$.
\end{lem} 
\begin{proof}
	Let $\gamma_0=(1,1,r)$, then we have $|\{\gamma\in[\lambda]:\gamma \leq \gamma_0\}|=k$ and $|\{\gamma\in[\mu]:\gamma \leq \gamma_0\}|=l$. The statement is now an immediate consequence of the definition of the partial order in $(\ref{podf})$.
\end{proof}
The converse of this lemma is not always true. Counterexamples can be found in Figure-\ref{fig-po}.

To study the cellular structure of our generalized Temperley-Lieb algebras, we are particularly interested in those multipartitions with at most two non-empty components. The following lemmas give descriptions of the partial order on the elements of $ \mathfrak{B}_n^{(r)}(1)$ and $ \mathfrak{B}_n^{(r)}(2)$.

\begin{lem}$\label{25}$
	Let $\lambda,\mu\in \mathfrak{B}_n^{(r)}(1)$. Suppose the $k^{th}$ component of $\lambda$ and $l^{th}$ component of $\mu$ are non-empty; then $\lambda\unlhd \mu$ if and only if $k\leq l$.
\end{lem}

\begin{lem}$\label{26}$
	Let $\lambda\in \mathfrak{B}_n^{(r)}(2)$ with $\lambda^{(k_1)}$ and $\lambda^{(k_2)}$ ($k_1<k_2$) non-empty and let $\mu\in \mathfrak{B}_n^{(r)}(1)$ with $\mu^{(l)}$ as the only non-empty component. Then $\lambda\unlhd \mu$ if and only if $k_1\leq l$.
\end{lem}

Denote by $\lambda_{k_1,k_2}^{[a]}\in \mathfrak{B}_n^{(r)}(2)$ ($k_1<k_2$, $|a|<n$ and $a\equiv n$ $mod$ 2 ) the multipartition $\lambda$ of which $\lambda^{(k_1)}$ and $\lambda^{(k_2)}$ are the two non-empty components and $|\lambda^{(k_1)}|-|\lambda^{(k_2)}|=a$. The following lemma gives a description of the partial order on $\mathfrak{B}_n^{(r)}(2)$.

\begin{lem}$\label{27}$
	Let $\lambda_{k_1,k_2}^{[a]},\mu_{l_1,l_2}^{[b]}\in \mathfrak{B}_n^{(r)}(2)$. Then $\lambda_{k_1,k_2}^{[a]}\unlhd \mu_{l_1,l_2}^{[b]}$ if and only if $k_1\leq l_1$, $k_2\leq l_2$ and one of the following holds:
	
	1. $|a|<|b|$;
	
	2.$|a|=|b|$ and $a\geq b$;
	
	3. $|a|=|b|$, $a<b$ and $k_2\leq l_1$.
	
\end{lem}
\begin{proof}
	For $i=1,2$, denote by $a_i$ and $b_i$ the number of nodes in $\lambda^{(k_i)}$ and $\mu^{(l_i)}$ respectively.
	
	We first check the sufficiency. Let $c=min(a_1,a_2,b_1,b_2)$. As $|a|\leq|b|$, $c=b_i$ for some $i\in\{1,2\}$. The condition $k_1\leq l_1$, $k_2\leq l_2$ implies that
\begin{equation}
	|\{\gamma\in[\lambda]:\gamma \leq \gamma_0\}|\geq 	|\{\gamma\in[\mu ]:\gamma \leq \gamma_0\}|,
\end{equation}
for $\gamma_0\in \{1,2,\dots,c\}\times \{1\}\times\{1,2\dots,r\}$ and the equality holds for  $\gamma_0=(c,1,r)$. Denote by $[\lambda']$ and $[\mu']$ the sub-diagrams of $[\lambda_{k_1,k_2}^{[a]}]$ and $[\mu_{l_1,l_2}^{[b]}]$	consisting of the nodes below the $c^{th}$ row. If they are not empty, $|a|\leq|b|$ implies that $[\mu']$ has exactly one non-empty component and $[\lambda']$ has at most two. Lemma $\ref{25}$ and $\ref{26}$ guarantees the sufficiency of the listed three cases.

    On the other hand, if $\lambda_{k_1,k_2}^{[a]}\unlhd \mu_{l_1,l_2}^{[b]}$, let $\gamma_0=(1,1,l_i)$ with $i=1,2$ in $(\ref{podf})$, then we have
    \begin{equation*}
    	k_1\leq l_1\text{ and }k_2\leq l_2.
    \end{equation*}
Similarly, let $c=min(a_1,a_2,b_1,b_2)$ and $\gamma_0=(c,1,r)$. We have 
\begin{equation}
	|\{\gamma\in[\lambda]:\gamma \leq \gamma_0\}|= 	|\{\gamma\in[\mu ]:\gamma \leq \gamma_0\}|.
\end{equation}
Denote by $[\lambda']$ and $[\mu']$ the sub-diagrams of $[\lambda_{k_1,k_2}^{[a]}]$ and $[\mu_{l_1,l_2}^{[b]}]$	consisting of the nodes below the $c^{th}$ row. Then $\lambda_{k_1,k_2}^{[a]}\unlhd \mu_{l_1,l_2}^{[b]}$ implies $\lambda'\unlhd \mu'$. If $|a|>|b|$, we have $\lambda'\in \mathfrak{B}_{n-2c}^{(r)}(1)$ and $\mu'\in \mathfrak{B}_{n-2c}^{(r)}(2)$, which is contradictory to $\lambda'\unlhd \mu'$ by Lemma $\ref{dw}$. If $|a|=|b|$ and $a<0<b$, then $\lambda'\in \mathfrak{B}_{n-2c}^{(r)}(1)(l_1)$ and $\mu'\in \mathfrak{B}_{n-2c}^{(r)}(1)(k_2)$. By Lemma $\ref{25}$, we have $k_2\leq l_1$. Otherwise, we have either $|a|<|b|$ or $|a|=|b|$ and $a\geq b$. This completes the proof.
\end{proof}

 We next introduce a special standard tableau $t^{\lambda}$ for each multipartition $\lambda$ which will be used in Section \ref{gcs} to define the cellular basis of our generalised Temperley-Lieb algebra $TL_{r,1,n}$. Let $t^{\lambda}$ be the tableau of shape $\lambda$ such that $t^{\lambda}(i)<t^{\lambda}(j)$ if $1\leq i<j\leq n$. As $\leq$ is a total order on the nodes, $t^{\lambda}$ is unique for $\lambda\in \mathfrak{B}_n^{(r)}$. Moreover, for each $t\in Std(\lambda)$, denote by $d(t)\in \mathfrak{S}_n$ the unique permutation such that $t=t^{\lambda}\circ d(t)$ where $\circ$ is the natural $\mathfrak{S}_n$-action on the tableaux.

\subsection{Cyclotomic Hecke algebras and their representations}
\label{DF}
We next recall the definition of cyclotomic Hecke algebras and description of their representations following \cite{ARIKI1994216} by Ariki and Koike. The irreducible representations are described in the language of multipartions. Our generalised Temperley-Lieb algebras will be defined as quotients of cyclotomic Hecke algebras by certain ideals generated by a set of idempotents corresponding to a set of irreducible representations.

We first recall the presentation of the imprimitive reflection group $G(r,1,n)$.
\begin{lem} (cf. \cite{broue1994complex},Appendix 2)
	Let $r$ and $n$ be positive integers. The unitary reflection group $G(r,1,n)$ is the group generated by $s_0,s_1,s_2,$ $s_3,\dots,s_{n-1}$ subject to the following relations:
	\begin{equation*}
	  	\begin{aligned}
		s_0^r=s_1^2=s_2^2&=\dots=s_{n-1}^2=1;\\
		s_0s_1s_0s_1=&s_1s_0s_1s_0;\\
		s_is_{i+1}s_i=s_{i+1}s_is_{i+1} & \text{ for }1\leq i\leq n-2;\\
		s_is_j=s_js_i & \text{ for }|i-j|\geq 2.
	\end{aligned}
	\end{equation*}
\end{lem}

The generators $s_0,s_1,s_2,\dots,s_{n-1}$ are called simple reflections and the relations above can be represented in the Dynkin diagram in Figure-$\ref{DDr}$.
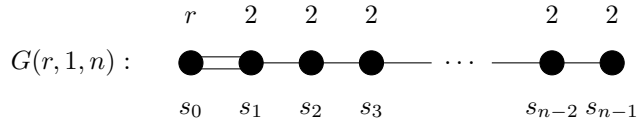
\begin{figure}[ht]
	\begin{center}
		\begin{tikzpicture}[scale=0.8]
			
			\foreach \x in {1,2,3,4,8,7}
			\filldraw(\x,0) circle (0.2cm);

			\draw (1,0.1)--(2,0.1);
			\draw (1,-0.1)--(2,-0.1);
			\draw (2,0)--(5,0);
			\draw (6,0)--(8,0);

			\draw node[below] at (1,-0.5){$s_0$};\draw node[below] at (2,-0.5){$s_1$};\draw node[below] at (3,-0.5){$s_2$};\draw node[below] at (4,-0.5){$s_3$};\draw node at (5.5,0){$\dots$};\draw node[below] at (7,-0.5){$s_{n-2}$};\draw node[below] at (8,-0.5){$s_{n-1}$};
			\draw node at (-1,0){$G(r,1,n):$};
			\draw node[above] at (1,0.5){$r$};
			\draw node[above] at (2,0.5){$2$};\draw node[above] at (3,0.5){$2$};\draw node[above] at (4,0.5){$2$};\draw node[above] at (7,0.5){$2$};\draw node[above] at (8,0.5){$2$};
			
		\end{tikzpicture}
	\end{center}
	\caption{Dynkin diagram of type $G(r,1,n)$}\label{DDr}
\end{figure} 

The reflection group $G(r,1,n)$ can be viewed as the group consisting of all $n\times n$ monomial matrices such that each non-zero entry is an $r^{th}$ root of unity. Moreover, if $\zeta \in \mathbb{C}^*$ is a primitive $r^{th}$ root of unity, the complex reflection group $G(r,1,n)$ is generated by the following elements:
\begin{equation*}
  \begin{aligned}
	s_0&=\zeta E_{1,1}+\sum_{k=2}^{n}E_{k,k},\\
	s_i&=\sum_{1\leq k\leq i-1}E_{k,k}+E_{i+1,i}+E_{i,i+1}+\sum_{i+2\leq k\leq n}E_{k,k} \text{ for all }1\leq i\leq n-1
\end{aligned}  
\end{equation*}

where $E_{i,j}$ is the elementary $n\times n$ matrix with the $(i,j)$-entry nonzero. Viewed thus, the generators are realised as (pseudo)-reflections in the $n$ dimensional space over $\mathbb{C}$, exhibiting $G(r,1,n)$ as a complex reflection group.

 Let $R=\mathbb{Z}[q,q^{-1},v_1,v_2,\dots,v_r]$ where $q,v_1,v_2,\dots,v_r$ are indeterminates over $\mathbb{C}$. The cyclotomic Hecke algebra corresponding to $G(r,1,n)$ over $R$ is defined as follows:
\begin{defn}\label{Hr}(\cite{broue1994complex},Definition 4.21)
	Let $G(r,1,n)$ and $R$ be as defined above. The cyclotomic Hecke algebra $H(r,1,n)$ corresponding to $G(r,1,n)$ over $R$ is the associative algebra generated by $T_0,T_1,\dots,T_{n-1}$ subject to the  following relations:
	\begin{align}
		(T_0-v_1)(T_0-v_2)\dots (T_0-v_r)&=0 ;\label{ordering relation}\\
		(T_i-q)(T_i+1)&=0\textit{ for } 1\leq i\leq n-1;\label{quadra rlt}\\
		T_0T_1T_0T_1&=T_1T_0T_1T_0;\\
		T_iT_{i+1}T_i&=T_{i+1}T_iT_{i+1}\textit{ for }1\leq i\leq n-2;\\
		T_iT_j&=T_jT_i \textit{ for }|i-j|\geq 2.
		\end{align}
\end{defn} 

The irreducible representations of $H(r,1,n)$ are described in the language of $r$-partitions of $n$.
Let $\lambda$ be an $r$-partition of $n$ and $K$ be the quotient field of $R$. Denote by $f_{\lambda}$ the number of standard tableaux of shape $\lambda$ and by $\mathfrak{Y}^{\lambda}=\{ \mathbb{Y}_1,\mathbb{Y}_2,\dots,\mathbb{Y}_{f_{\lambda}}\}$ the set of all the standard tableaux of shape $\lambda$. Let $\mathbb{V}_{\lambda}$ be the vector space over $K$ with the basis $\mathbb{Y}_1,\mathbb{Y}_2,\dots,\mathbb{Y}_{f_{\lambda}}$; namely,
\begin{equation}
	\mathbb{V}_{\lambda}=\bigoplus_{i} K\mathbb{Y}_i,
\end{equation}
where the sum runs over all the standard tableaux $\mathbb{Y}_i$ of shape $\lambda$.

Let $\gamma=(a,b,i)$ be a node in $\lambda^{(i)}$ for some $i$. The content $c(\lambda :\gamma)$ of $\gamma$ in $\lambda$ is defined to be the content of $\gamma$ in the Young diagram $\lambda^{(i)}$, namely, $c(\lambda :\gamma)=c(\lambda^{(i)} :\gamma)=b-a$.

For any indeterminate $x$ and an integer $k$, let $\Delta(k,x)$ be the polynomial $\Delta(k,x)=1-q^{k}x$ in $q$ and $x$ with coefficients in $K$ and define $M(k,x)$ as follows:
\begin{equation} \label{dfm}
	M(k,x)=\dfrac{1}{\Delta(k,x)}\left(
	\begin{matrix}
		& q-1 &\Delta(k+1,x)\\
		& q\Delta(k-1,x) &q^{k}(1-q)x
	\end{matrix} \right).
\end{equation}
With the notations above, Ariki and Koike describe a complete set of the irreducible representations of $H(r,1,n)$:
\begin{thm}
\label{AK}
	(\cite{ARIKI1994216},Theorem 3.10)
	
	Let $K$ be the quotient field of the base ring $R$ which is defined before Definition \ref{Hr}. $H(K)=H(r,1,n)\bigotimes K$ is semisimple over $K$. The set $\{\mathbb{V}_{\lambda} | \lambda \vdash n \}$ is a complete set of non-isomorphic irreducible $H(K)$-modules. For each multipartition $\lambda$ and standard tableau $\mathbb{Y}$ of shape $\lambda$, $H(r,1,n)$ acts on $\mathbb{Y} \in \mathbb{V}_{\lambda}$ as follows:
	
	(1) The action of $T_0$.
	\begin{equation}
		T_0 \mathbb{Y}=v_{\tau(1)} \mathbb{Y}
	\end{equation}
	where the number 1 occurs in the $\tau(1)$-th component, $\lambda^{(\tau(1))}$, of the tableau $\mathbb{Y}$.
	
	(2)The action of $T_i(i=1,2,\dots,n-1)$.
	
	Case (2.1). If $i$ and $i+1$ appear in the same row of the same component of $\mathbb{Y}$, then
	\begin{equation}
		T_i\mathbb{Y}=q\mathbb{Y}.
	\end{equation}

	Case (2.2). If $i$ and $i+1$ appear in the same column of the same component of $\mathbb{Y}$, then
	\begin{equation}
		T_i\mathbb{Y}=-\mathbb{Y}.
	\end{equation}

	Case (2.3). $i$ and $i+1$ appear neither in the same row nor in the same column of  $\mathbb{Y}$. Let $\mathbb{Y}'$ be the standard tableau obtained by transposing the letters $i$ and $i+1$ in $\mathbb{Y}$. The action of $T_i$ on the two-dimensional subspace spaned by $\mathbb{Y}$ and $\mathbb{Y}'$ is given by
	\begin{equation}
		[T_i\mathbb{Y},T_i\mathbb{Y}']=[\mathbb{Y},\mathbb{Y}'] M(c(\lambda: i)-c(\lambda: i+1),\dfrac{v_{\tau(i)}}{v_{\tau(i+1)}}),
	\end{equation}
	where $M(c(\lambda: i)-c(\lambda: i+1),\dfrac{v_{\tau(i)}}{v_{\tau(i+1)}})$ is the $2\times 2$ matrix $M(k,x)$ defined in (\ref{dfm}) with $k=c(\lambda: i)-c(\lambda: i+1)$ and $x=\dfrac{v_{\tau(i)}}{v_{\tau(i+1)}}$ and the numbers $i$ and $i+1$ occur in the $\tau(i)$-th and $\tau(i+1)$-th components of the tableau $\mathbb{Y}$, respectively.
\end{thm}
Further, if $R$ is specialised to a field $F$, the following theorem states exactly when $H(r,1,n)$ is semisimple:

\begin{thm} \label{ssth}
    (\cite{Ariki1994ss}, main theorem) Suppose $R$ is a field $F$. Then the cyclotomic Hecke algebra $H(r,1,n)$ is semisimple if and only if the parameters $q,v_1,v_2,\dots,v_r$ satisfy
    \begin{equation*}
        q^d v_i\neq v_j
    \end{equation*}
    for $i\neq j$ and $d\in\mathbb{Z}$ with $|d|<n$ and
    \begin{equation*}
        1+q+q^2+\dots+q^j\neq 0
    \end{equation*}
    for $1\leq j\leq n-1$.
\end{thm}

In the case when the cyclotomic Hecke algebra $H(r,1,n)$ is semisimple, each one-dimensional module corresponds to a central idempotent in $H(r,1,n)$. We will use these idempotents to define our generalised Temperley-Lieb algebra $TL_{r,1,n}$ in the next section.

\subsection{The cyclotomic KLR algebras}$\label{2.1}$
To show the graded cellularity of our generalised Temperley-Lieb algebras, we shall make use of the cyclotomic KLR algebras, which are isomorphic to the cyclotomic Hecke algebras $H(r,1,n)$. In the next section, we will use this isomorphism to give a second interpretation of our generalised Temperley-Lieb algebra $TL_{r,1,n}$. 

The KLR algebra was first introduced independently by Khovanov and Lauda in \cite{khovanov2008diagrammatic} and by Rouquier in \cite{rouquier20082kacmoody} to categorify the quantum deformation of the universal enveloping algebra of the lower-triangular subalgebra of a Kac-Moody algebra $\mathfrak{g}$. Brundan and Kleshchev have shown in \cite{Brundan_2009} that the KLR algebras of type $A_{n-1}$ are isomorphic to the cyclotomic Hecke algebras corresponding to the complex reflection groups $G(r,1,n)$ or their rational degenerations.
In this section, we recall the special cyclotomic KLR algebra which is isomorphic to $H(r,1,n)$. We refer readers to \cite{Brundan_2009} for a broader discussion.

Let $F$ be a fixed ground field and $q \in F^*$. Let $e$ be the smallest positive integer such that $1+ q + \dots+ q ^{e-1} = 0$, setting $e := 0$ if no such integer exists. Let $\Gamma$ be the quiver with vertex set $I := \mathbb{Z}/e\mathbb{Z}$, and a directed edge from $i$ to $j$ if $j = i + 1$. Thus $\Gamma$ is the quiver of type $A_{\infty}$ if $e=0 $ or $A_{e-1}^{(1)}$, with the following orientation:

\begin{tikzcd}
  A_{\infty}   :   &  \dots   \arrow[r ] & -2 \arrow[r ] & -1\arrow[r ] &   0  \arrow[r ] &   1   \arrow[r ] &    2 \arrow[r ] &    \dots\\
 A_{e-1}^{(1)} : & 0\leftrightarrow 1 & & 0  \arrow[dl] & & 0 \arrow[r] & 1 \arrow[d]
&\dots\\
 & &1 \arrow{rr} & &\arrow[ul]  2&  
3 \arrow[r] \arrow[u] 
& 2
 \end{tikzcd}

Let $(c_{ij})$ be the Cartan matrix associated with $\Gamma$, so that
\begin{gather}
	c_{ij}=
	\begin{cases}
		2 &\textit{ if } i=j;\\
		0 &\textit{ if } i\neq j\pm 1;\\
		-1 & \textit{ if } e\neq 2\textit{ and } i=j\pm1;\\
		-2 & \textit{ if } e= 2\textit{ and } i=j\pm1.
	\end{cases}
\end{gather}
Let $\{\alpha_i|i\in I\}$ be the corresponding set of simple roots and $\{\Lambda_i|i\in I\}$ the fundamental weights. Let $( \;,\; )$ be the bilinear form determined by
\begin{gather*}
	(\alpha_i,\alpha_j)=c_{ij} \textit{ and }(\Lambda_i,\alpha_j)=\delta_{ij}.
\end{gather*}
Let $P_+=\oplus_{i\in I}\mathbb{N}\Lambda_i$ be the dominant weights in the weight lattice and  $Q_+=\oplus_{i\in I}\mathbb{N}\alpha_i$ be the positive roots. For $\Lambda\in P_+$, define the length of $\Lambda$, $l(\Lambda):=\sum_{i\in I}(\Lambda,\alpha_i)$ and for $\alpha\in Q_+$, let the height of $\alpha$ be $ht(\alpha):=\sum_{i\in I}(\alpha,\Lambda_i)$. Fix an $\alpha \in Q_+$, define $I^{\alpha}$ as follows:
\begin{equation}
	I^{\alpha}:=\{i\in I^{ht(\alpha)}|\alpha=\alpha_{i_1}+\alpha_{i_2}+\dots+\alpha_{i_{ht(\alpha)}}\}.
\end{equation}

Fixing a dominant weight $\Lambda$, we next recall the definition of the cyclotomic KLR algebra of type $A_{n-1}$ associated with the weight $\Lambda$. The following definition is originally from the work of Khovanov and Lauda in \cite{khovanov2008diagrammatic} and Rouquier in \cite{rouquier20082kacmoody}.

\begin{defn} $\label{klrdf}$
	The cyclotomic KLR algebra of type $A_{n-1}$ associated with the weight $\Lambda$ is defined as $\mathcal{R}_n^{\Lambda}:=\oplus_{ht(\alpha)=n}\mathcal{R}_{\alpha}^{\Lambda}$, where $\mathcal{R}_{\alpha}^{\Lambda}$ is the unital associative $F$-algebra with generators
	\begin{equation}
		\{\psi_1,\dots,\psi_{n-1}\}\cup\{y_1,y_2,\dots,y_n\}\cup \{e(i)|i\in I^{\alpha}\}
	\end{equation}
	subject to  the following relations
	\begin{equation}
		y_1^{(\Lambda,\alpha_{i_1})}e(i)=0,
	\end{equation}
	\begin{equation}
		e(i)e(j)=\delta_{ij}e(i),
	\end{equation}
	\begin{equation}
		\sum_{i\in I^{\alpha}}e(i)=1,
	\end{equation}
	\begin{equation} \label{18eq}
		e(i)y_r=y_r e(i),
	\end{equation} 
	\begin{equation} 
		\psi_r e(i)=e(s_r\circ i)\psi_r \label{sym},
	\end{equation} 
	\begin{equation}
		y_r y_s=y_sy_r,
	\end{equation} 
	\begin{equation}\label{21eq}
		\psi_r y_s=y_s \psi_r \textit{ if }s\neq r,r+1;
	\end{equation}
	\begin{equation} \label{22eq}
		\psi_r \psi_s=\psi_s \psi_r \textit{ if }|r-s|\neq 1;
	\end{equation}
	\begin{gather} \label{23eq}
		\psi_r y_{r+1}e(i)=
		\begin{cases}
			y_r \psi_r e(i)+e(i) &\textit{ if } i_r= i_{r+1};\\
			y_r \psi_r e(i) &\textit{ if } i_r\neq i_{r+1};
		\end{cases}
	\end{gather}
	\begin{gather} \label{24eq}
		y_{r+1} \psi_r e(i)=
		\begin{cases}
			\psi_r y_r e(i)+e(i) &\textit{ if } i_r= i_{r+1};\\
			\psi_r y_r e(i) &\textit{ if } i_r\neq i_{r+1};
		\end{cases}
	\end{gather}
	\begin{gather} \label{25eq}
		\psi_r^2e(i)=
		\begin{cases}
			0 &\textit{ if } i_r= i_{r+1};\\
			e(i) &\textit{ if } i_r\neq i_{r+1}\pm 1;\\
			(y_{r+1}-y_r)e(i) &\textit{ if } i_r\rightarrow i_{r+1};\\
			(y_r-y_{r+1})e(i) &\textit{ if } i_r\leftarrow i_{r+1};\\
			(y_{r+1}-y_r)(y_r-y_{r+1})e(i) &\textit{ if } i_r\leftrightarrow i_{r+1}.
		\end{cases}
	\end{gather}
	\begin{gather} \label{26eq}
		(\psi_r\psi_{r+1}\psi_r-\psi_{r+1}\psi_r\psi_{r+1})e(i)=
		\begin{cases}
			e(i) &\textit{ if } i_{r+2}=i_r\rightarrow i_{r+1};\\
			-e(i) &\textit{ if } i_{r+2}=i_r\leftarrow i_{r+1};\\
			(y_r-2y_{r+1}+y_{r+2})e(i) &\textit{ if } i_{r+2}=i_r\leftrightarrow i_{r+1};\\
			0 &\textit{ otherwise } .
		\end{cases}
	\end{gather}
	where $s_r\circ i$ in (\ref{sym}) is the natural action of $\mathfrak{S}_n$ on $I^n$ and the arrows in (\ref{25eq}) and (\ref{26eq}) are those in the quiver $\Gamma$.
\end{defn}

Note that all the relations above are homogeneous with respect to the following degree function on the generators:
\begin{equation}\label{degf}
	\text{deg } e(i)=0,\text{deg } y_r=2\text{ and  deg }\psi_se(i)=-c_{i_s,i_{s+1}}
\end{equation}
where $i\in I^n$,$1\leq r\leq n$ and $1\leq s\leq n-1$. Therefore, the cyclotomic KLR algebra defined above is $\mathbb{Z}$-graded with respect to the degree function in (\ref{degf}).

Let $\mathcal{H}_n$ be the affine Hecke algebra over $F$ corresponding to $\mathfrak{S}_n$, that is the $F$-algebra generated by $T_1,T_2,\dots,T_{n-1},X_1^{\pm1},\dots,X_n^{\pm1}$ subject to the following relations:
\begin{equation}
	X_rX_s=X_sX_r,qX_{r+1}=T_rX_rT_r,
\end{equation}
\begin{equation}
	T_rX_s=X_sT_r \text{ if }s\neq r,r+1,
\end{equation}
\begin{equation}
	(T_r-q)(T_r+1)=0, T_rT_{r+1}T_r=T_{r+1}T_rT_{r+1},
\end{equation}
\begin{equation}
	T_rT_s=T_sT_r \text{ if }|r-s|>1.
\end{equation}

Let $\mathcal{H}_n^{\Lambda}$ be the cyclotomic quotient of $\mathcal{H}_n$ corresponding to the weight $\Lambda$, that is
\begin{equation} \label{32eq}
	\mathcal{H}_n^{\Lambda}:=\mathcal{H}_n/\langle \prod_{i\in I}(X_1-q^i)^{(\Lambda,\alpha_i)}\rangle.
\end{equation}

Suppose that $M$ is a finite dimensional $\mathcal{H}_n^{\Lambda}$-module. Then $M$ may be regarded as a $\mathcal{H}$ module, and by \cite[Lemma 4.7]{Grojnowski1999AffineSC}, when $q\neq 1$, the eigenvalues of $X_r$ (for all $r$) are of form $q^i$ where $i\in I$. That implies $M$ decomposes as a direct sum of the joint generalised eigenspaces for $X_1,X_2,\dots,X_n$, that is $M=\oplus_{i\in I^n}M_i$ with
\begin{equation}
	M_i:=\{v\in M|(X_r-q^{i_{r}})^{dim(M)}v=0 ,\text{ for }1\leq r\leq n \}. \label{es}
\end{equation}
In particular, let $M$ be the regular $\mathcal{H}_n^{\Lambda}$ module. Then there is a set of pairwise orthogonal idempotents $\{e(i)|i\in I^n\}$ in $\mathcal{H}_n^{\Lambda}$ such that $M_i=e(i)M$. All but finitely many of the $e(i)$ are zero and their sum is the identity in $\mathcal{H}_n^{\Lambda}$. Fix an $\alpha\in Q_+$, let
\begin{equation}
	e_{\alpha}:=\sum_{i\in I^{\alpha}}e(i)\in \mathcal{H}_n^{\Lambda}. \label{ce}
\end{equation}
As a consequence of \cite{LYLE2007854} or \cite[Theorem 1]{Brundan2006CentersOD}, $e_{\alpha}$ is either zero or a primitive central idempotent in $\mathcal{H}_n^{\Lambda}$. Let
\begin{equation}
	\mathcal{H}_{\alpha}^{\Lambda}:=e_{\alpha}\mathcal{H}_n^{\Lambda}.
\end{equation}
Then $\mathcal{H}_{\alpha}^{\Lambda}$ is either zero or an indecomposable two-sided ideal of $\mathcal{H}_n^{\Lambda}$.

We will next introduce the isomorphism between $\mathcal{R}_{\alpha}^{\Lambda}$ and $\mathcal{H}_{\alpha}^{\Lambda}$ constructed by Brundan and Kleshchev in \cite{Brundan_2009} by defining elements in $\mathcal{H}_{\alpha}^{\Lambda}$ which satisfy the relations in $\mathcal{R}_{\alpha}^{\Lambda}$. Let $e(i)$ be the idempotent defined above. For $1\leq r\leq n$, define
\begin{equation}
	y_r:=\sum_{i\in I^{\alpha}}(1-q^{-i_r}X_r)e(i).
\end{equation}
And for $1\leq r\leq n$ and $i\in I^{\alpha}$ set
\begin{equation}
	y_r(i):=q^{i_r}(1-y_r)\in F[y_1,y_2,\dots,y_n].
\end{equation}
Define power series $P_r(i),Q_r(i)\in F[[y_r,y_{r-1}]]$ as follows:
\begin{gather}
	P_r(i):=
	\begin{cases}
		1 &\textit{ if } i_r= i_{r+1};\\
		(1-q)(1-y_r(i)y_{r+1}(i)^{-1})^{-1} &\textit{ if } i_r\neq i_{r+1}.
	\end{cases}
\end{gather}

\begin{gather}
	Q_r(i):=
	\begin{cases}
		1-q+qy_{r+1}-y_r &\text{ if } i_r= i_{r+1};\\
		(y_r(i)-qy_{r+1}(i))/ (y_r(i)-y_{r+1}(i))&\text{ if } i_r\neq i_{r+1}\pm 1;\\
		(y_r(i)-qy_{r+1}(i))/ (y_r(i)-y_{r+1}(i))^2 &\text{ if } i_r\rightarrow i_{r+1};\\
		q^{i_r} &\text{ if } i_r\leftarrow i_{r+1};\\
		q^{i_r}/ (y_r(i)-y_{r+1}(i)) &\text{ if } i_r\leftrightarrow i_{r+1}.
	\end{cases}
\end{gather}

Then for $1\leq r\leq n-1$, set 
\begin{equation}
	\psi_r:=\sum_{i\in I^{\alpha}}(T_r+P_r(i))Q_r(i)^{-1} e(i).
\end{equation}

We do not distinguish here between the generators of the cyclotomic KLR algebra and the corresponding elements in the cyclotomic Hecke algebra because of the following isomorphism theorem by Brundan and Kleshchev:
\begin{thm}$\label{iso thm}$
	(\cite[Theorem 1.1]{Brundan_2009}) Let $\alpha$ be an element in the positive root lattice $Q_+$, such that the two-sided ideal $\mathcal{H}_{\alpha}^{\Lambda}\neq 0$. The map $f: \mathcal{R}_{\alpha}^{\Lambda}\to \mathcal{H}_{\alpha}^{\Lambda}$ such that
	\begin{equation*}
		f(e(i))=e(i),f(y_r)=y_r\text{ and }f(\psi_s)=\psi_s
	\end{equation*}
	extends uniquely as an algebra isomorphism.
\end{thm}

We identify $\mathcal{R}_{\alpha}^{\Lambda}$ and $\mathcal{H}_{\alpha}^{\Lambda}$ via $f$ in the rest of this paper. Moreover, the inverse of $f$ is given by
\begin{equation}
	f^{-1}(X_r)=\sum_{i\in I^{\alpha}}y_r(i)e(i)
\end{equation}
and
\begin{equation} \label{isoinv}
	f^{-1}(T_r)=\sum_{i\in I^{\alpha}}(\psi_rQ_r(i)-P_r(i))e(i).
\end{equation}
The cyclotomic Hecke algebra $\mathcal{H}_n^{\Lambda}$ is naturally a subalgebra of $\mathcal{H}_{n+1}^{\Lambda}$ and $\mathcal{H}_{n+1}^{\Lambda}$ can also be regarded as a free $\mathcal{H}_n^{\Lambda}$-module. This implies that we can define the usual restriction and induction functors between $Rep(\mathcal{R}_n^{\Lambda})$ and $Rep(\mathcal{R}_{n+1}^{\Lambda})$. These functors can be decomposed into $i$-restrictions $Res_i^{\Lambda}$ and $i$-inductions $Ind_i^{\Lambda}$, for $i\in I$, by projecting onto the blocks $\mathcal{R}_{\alpha}^{\Lambda}$ and $\mathcal{R}_{\alpha+\alpha_i}^{\Lambda}$ where $ht(\alpha)=n$. For each $i\in I$, define
\begin{equation}
	e_{n,i}:=\sum_{j\in I^n}e(j,i).
\end{equation}  
Let $f_i:\mathcal{R}_n^{\Lambda}\to \mathcal{R}_{n+1}^{\Lambda}$ be the embedding map given by:
\begin{equation}
	f_i(e(j))=e(j,i),f_i(y_r)=e_{n,i}y_r\text{ and } f_i(\psi_s)=e_{n,i}\psi_s,
\end{equation}
where $j\in I^n$, $1\leq r\leq n$ and $1\leq s\leq n-1$. It should be observed that this embedding is not unital.
The restriction and induction functors induced by $f_i$ are:
\begin{equation}
	\begin{aligned}
		Res_i^{\Lambda}:Rep(\mathcal{R}_{n+1}^{\Lambda}) &\mapsto Rep(\mathcal{R}_{n}^{\Lambda}) \\
		N&\mapsto Ne_{n,i}\\
		Ind_i^{\Lambda}:Rep(\mathcal{R}_{n}^{\Lambda}) &\mapsto Rep(\mathcal{R}_{n+1}^{\Lambda})\\
		M&\mapsto M\otimes_{\mathcal{R}_{n}^{\Lambda}} e_{n,i}\mathcal{R}_{n+1}^{\Lambda}. \label{EF}
	\end{aligned}
\end{equation}
The following lemma indicates how to extend a certain ideal $\mathcal{I}_n$ of $\mathcal{R}_{n}^{\Lambda}$ to an ideal of $\mathcal{R}_{n+1}^{\Lambda}$.

\begin{lem}
	Let $\alpha \in Q^+$ be an element of height $n$ and $\mathcal{I}_n=\mathcal{R}_{\alpha}^{\Lambda}$ be an indecomposable two-sided ideal of $\mathcal{R}_{n}^{\Lambda}$.
   Denote by $\mathcal{I}_{n+1}$ the two-sided ideal of $\mathcal{R}_{n+1}^{\Lambda}$ generated by $\mathcal{I}_n$. Then we have
   \begin{equation*}
   	\mathcal{I}_{n+1}=\langle \sum_{\mathbf{i}\in I^{\alpha},i\in I} e(\mathbf{i},i) \rangle_{\mathcal{H}_{n+1}^{\Lambda}}.
   \end{equation*}
	$\label{keylm}$
\end{lem}
\begin{proof}
	By Theorem $\ref{iso thm}$, $\mathcal{R}_{\alpha}^{\Lambda}\cong \mathcal{H}_{\alpha}^{\Lambda}$. We can identify $\mathcal{R}_{\alpha}^{\Lambda}$ and $ \mathcal{H}_{\alpha}^{\Lambda}$. For $\mathbf{i}\in I^{n}$ ($\mathbf{j}\in I^{n+1}$), let $e(\mathbf{i})$ ($e(\mathbf{j})$) be the idempotent in $\mathcal{H}_{n}^{\Lambda}$ ($\mathcal{H}_{n+1}^{\Lambda}$) which corresponds to the generalised eigenspaces $M_{\mathbf{i}}$ ($N_{\mathbf{j}}$) in ($\ref{es}$) of the regular left module. We first prove the following equation:
	\begin{equation}
		e(\mathbf{i})e(\mathbf{j})=
		\begin{cases}
			e(\mathbf{j}) &\textit{ if }\mathbf{j}=(\mathbf{i},i) \textit{ for some }i\in I;\\
			0 &\textit{otherwise}.
		\end{cases}\label{orth}
	\end{equation}
	Let $\mathbf{j}=(j_1,j_2,\dots,j_{n+1})$ and $\mathbf{j}^{(n)}=(j_1,j_2,\dots,j_{n})$. By the definition of $e(\mathbf{j})$, $(X_r-q^{j_{r}})^{dim(\mathcal{H}_{n+1}^{\Lambda})}e(\mathbf{j})=0 ,\text{ for }1\leq r\leq n+1$. Regard $\mathcal{H}_{n+1}^{\Lambda}$ as a left $\mathcal{H}_{n}^{\Lambda}$-module, $M^{(n+1)}$. Then $e(\mathbf{j})\in M_{\mathbf{j}^{(n)}}^{(n+1)}$. By definition, $M_{\mathbf{j}^{(n)}}^{(n+1)}=e({\mathbf{j}^{(n)}})M^{(n+1)}$. So $e(\mathbf{j})=e({\mathbf{j}^{(n)}})h$ for some $h\in \mathcal{H}_{n+1}^{\Lambda}$. As $e(\mathbf{i})$'s are pairwise orthogonal idempotents, $e(\mathbf{i})e({\mathbf{j}^{(n)}})$ is $e(\mathbf{j}^{(n)})$ (when $\mathbf{i}=\mathbf{j}^{(n)}$) or zero, which implies ($\ref{orth}$).

	 Let $e_{\alpha}\in \mathcal{H}_{n}^{\Lambda}$ be the primitive central idempotent such that $\mathcal{H}_{\alpha}^{\Lambda}=e_{\alpha}\mathcal{H}_{n}^{\Lambda}$. Then we have 
\begin{equation*}
        \begin{aligned}
    	\mathcal{I}_{n+1}&=\langle \mathcal{I}_n \rangle_{\mathcal{R}_{n+1}^{\Lambda}}\\
    	&=\langle e_{\alpha}\mathcal{H}_{n}^{\Lambda} \rangle_{\mathcal{H}_{n+1}^{\Lambda}}\\
    	&=\langle e_{\alpha} \rangle_{\mathcal{H}_{n+1}^{\Lambda}}\\
    	&= \langle \sum_{\mathbf{i}\in I^{\alpha}} e(\mathbf{i})\sum_{\mathbf{j}\in I^{n+1}} e(\mathbf{j}) \rangle_{\mathcal{H}_{n+1}^{\Lambda}}\\
    	&= \langle \sum_{\mathbf{i}\in I^{\alpha},i\in I} e(\mathbf{i},i) \rangle_{\mathcal{H}_{n+1}^{\Lambda}}.
    \end{aligned}
\end{equation*}

    \end{proof}

\section{Generalised Temperley-Lieb algebras $TL_{r,1,n}$}
In this section, we define our generalised Temperley-Lieb algebra $TL_{r,1,n}$ as a quotient of the corresponding cyclotomic Hecke algebra $H(r,1,n)$. We will show that the Temperley-Lieb algebras of types $A_{n-1}$ and $B_n$ are both special cases of $TL_{r,1,n}$. We will also give an interpretation of our generalised Temperley-Lieb algebra with respect to the KLR generators. This interpretation will be used to construct the graded cellular structure of $TL_{r,1,n}$ in the next section.
In what follows, we require $R$ to be a field of characteristic 0 and $q,v_1,v_2,\dots,v_r\in R^*$ such that 
\begin{equation}\label{pmres}
    \frac{v_i}{v_j}\neq 1,q\textbf{ or } q^2
\end{equation}
for all $i\neq j$ and
\begin{equation}\label{pmres1}
    (1+q)(1+q+q^2)\neq 0.
\end{equation}
These restrictions on $q,v_1,v_2,\dots,v_r\in R^*$ are equivalent to the semi-simplicity of the parabolic subalgebras $H_{i,i+1}$ mentioned in the introduction by Theorem \ref{ssth}.
\subsection{The definition of $TL_{r,1,n}$ as a quotient of the cyclotomic Hecke algebra} \label{3.1}
Let $H(r,1,n)$ be the cyclotomic Hecke algebra defined in section $\ref{DF}$ and $H_{i,i+1}$ be the subalgebra of $H(r,1,n)$ which is generated by two non-commuting generators, $T_{i}$ and $T_{i+1}$ where $0\leq i\leq n-2$. We call $H_{i,i+1}$ a parabolic subalgebra of $H(r,1,n)$. There are two types of parabolic subalgebras:

Case 1. $i=0$. $H_{0,1}$ is a cyclotomic Hecke algebra corresponding to $G(r,1,2)$. By Theorem \ref{ssth}, it is semisimple given the restriction on parameters in $(\ref{pmres})$ and $(\ref{pmres1})$. It has $2r$ 1-dimensional representations corresponding to the multipartitions $(0,0,\dots,(2),\dots,0)$ and $(0,0,\dots,(1,1),\dots,0)$ where $(2)$ and $(1,1)$ are in the $j^{th}$ component of the $r$-partition with $1\leq j \leq r$. Denote by $E_0^{(j)}$ and $F_0^{(j)}$ the corresponding primitive central idempotents in $H_{0,1}$.
	
Case 2. $1\leq i\leq n-2$. $H_{i,i+1}$ is a Hecke algebra of type $A_2$. It has 2 1-dimensional representations and the corresponding primitive central idempotents are
\begin{equation*}
  \begin{aligned}
	E_i&=a(T_iT_{i+1}T_i+T_iT_{i+1}+T_{i+1}T_i+T_i+T_{i+1}+1);\\
	F_i&=b(T_iT_{i+1}T_i-qT_iT_{i+1}-qT_{i+1}T_i+q^2T_i+q^2T_{i+1}-q^3),
\end{aligned}  
\end{equation*}

where $a=(q^3+2q^2+2q+1)^{-1}$ and $b=-a$.
As a generalisation of the original Temperley-Lieb algebra, $TL_{r,1,n}$ is defined as the quotient of $H(r,1,n)$ by the two-sided ideal generated by half of the central idempotents listed above, that is:

\begin{defn}\label{TL}
	Let $H(r,1,n)$ be the cyclotomic Hecke algebra defined in \ref{Hr} where $R$ is a field of characteristic 0 and $q,v_1,v_2,\dots,v_r\in R^*$ satisfy (\ref{pmres}) and (\ref{pmres1}).
	Let $H_{i,i+1}$ be the parabolic subalgebra generated by $T_i$ and $T_{i+1}$ and $E_i$( $E_i^{(j)}$, if $i=0$) be the primitive central idempotents of $H_{i,i+1}$ listed above. The generalised Temperley-Lieb algebra $TL_{r,1,n}$ is
	\begin{equation*}
		TL_{r,1,n}:=H(r,1,n)/\langle E_0^{(1)},\dots,E_0^{(r)},E_1,\dots,E_{n-2}\rangle.
	\end{equation*}
    
\end{defn}
It should be remarked that the generalised Temperley-Lieb algebra depends heavily on the parameters $q,v_1,v_2,\dots,v_r$. To make the notations simpler, we omit the parameters without causing confusion.
The next two examples show that the usual Temperley-Lieb algebra $TL_n(q)$ and the blob algebra $TLB_n(q,Q)$ are special cases of $TL_{r,1,n}$ where $r=1,2$.
\begin{exmp}
	$r=1$. In this case, $T_0$ is an element in $R$, thus $T_0$ and $T_1$ commutative. So $H_{0,1}$ is no longer a parabolic subalgebra in this case. Therefore, $TL_{1,n}=H_{n-1}/\langle E_1,\dots,E_{n-2}\rangle=H_{n-1}/\langle E_1\rangle=TL_{n-1}(q)$.
\end{exmp}

\begin{exmp}
  $r=2$. Let $v_1=Q$ and $v_2=-Q^{-1}$. We have $H(2,1,n)=HB_n(Q,q)$. We next show $TL_{2,n}=TLB(Q,q)$ when $1+Q^2\in R^*$. 
  In Lemma 5.1 of \cite{GRAHAM2003479}, Graham and Lehrer point out that $C_0X=XC_0=u^{-2}Q^{-2}E_0^{(1)}$ which implies $E_0^{(1)}\in \langle X\rangle$ as $u,Q$ are invertible. We observe that
\begin{equation*}
  \begin{aligned}
	&E_0^{(1)}-Q^4E_0^{(2)}\\
	=&(u^2Q^2-u^2Q^2)T_2T_0T_2T_0+(u^2Q+u^2Q^3)T_2T_0T_2+(uQ^2-uQ^2)T_0T_2T_0\\&+(uQ+uQ^3)(T_0T_2+T_2T_0)+(Q+Q^3)T_0+(u-uQ^4)T_2+1-Q^4\\
	=&(u^2Q+u^2Q^3)T_2T_0T_2+(uQ+uQ^3)(T_0T_2+T_2T_0)+(Q+Q^3)T_0\\
	&+(u-uQ^4)T_2+1-Q^4\\
	=&(1+Q^2)(u^2QT_2T_0T_2+uQT_0T_2+uQT_2T_0+QT_0+u(1-Q^2)T_2+1-Q^2)\\
	=&(1+Q^2)(uT_2+1)(QT_0+1)(uT_2+1)-(1+Q^2)(u^2+u^3T_2+uQ^2T_2+Q^2)\\
	=&-u^2Q(1+Q^2)C_2C_0C_2+u^2Q(1+Q^2)\kappa C_1\\
	=&-u^2Q^2(1+Q^2)X.
\end{aligned}  
\end{equation*}

As $Q,u$ are invertible in $R$, we have:
\begin{equation}
	\langle (1+Q^2)X \rangle\subseteq\langle E_0^{(1)},E_0^{(2)} \rangle.
\end{equation}
On the other hand, the equation above implies
\begin{equation*}
    Q^4E_0^{(2)}=E_0^{(1)}+u^2Q^2(1+Q^2)X.
\end{equation*}
So we have
\begin{equation*}
    \langle E_0^{(1)},E_0^{(2)} \rangle \subseteq 	\langle E_0^{(1)},(1+Q^2)X \rangle \subseteq  \langle X\rangle.
\end{equation*}
Since $1+Q^2\in R^*$, we have $\langle X\rangle = \langle E_0^{(1)},E_0^{(2)}\rangle$.
  So we have 
  \begin{equation*}
  \begin{aligned}
	TL_{2,n}&=H(2,1,n)/\langle E_0^{(1)},E_0^{(2)},E_1,\dots,E_{n-2}\rangle\\
	&=HB_n(Q,q)/\langle E_0^{(1)},E_0^{(2)},E_1\rangle\\
	&=HB_n(Q,q)/\langle X, E_1\rangle\\
	&=TLB_n(Q,q).
\end{aligned}      
  \end{equation*}

\end{exmp}

We next show that the generalised Temperley-Lieb algebra $TL_{r,1,n}$ is a quotient of the corresponding generalised blob algebra $B(r,n)$. Let $\mathcal{I}_{1,n}$ be the two-sided ideal of $H(r,1,n)$ generated by $E_0^{(i)}$ for all $i=1,2,\dots,r$. The generalised blob algebra was originally defined by Martin and Woodcock in \cite{martin_woodcock_2003} as follows:
\begin{defn}
	(\cite{martin_woodcock_2003}, Section 3) Let $\mathcal{I}_{1,n}$ be the two-sided ideal of $H(r,1,n)$ generated by $E_0^{(i)}$ for all $i=1,2,\dots,r$. The generalised blob algebra $B_{r,n}$ is defined as the quotient of $H(r,1,n)$ by $\mathcal{I}_{1,n}$, that is
	\begin{equation}
	    B_{r,n}=H(r,1,n)/\mathcal{I}_{1,n}.
	\end{equation}
\end{defn}

Denote by $\mathcal{I}_{2,n}$ the ideal of $B_{r,n}$ generated by the representatives of $E_1,\dots,E_{n-2}$. The following lemma is a direct consequence of definitions:
\begin{lem}
   $TL_{r,1,n}\cong  B_{r,n}/\mathcal{I}_{2,n}$.
   
\end{lem}

\subsection{Realisation of $TL_{r,1,n}$ as a quotient of the cyclotomic KLR algebra}
We next give an interpretation of the generalised Temperley-Lieb algebra $TL_{r,1,n}$ as a quotient of the cyclotomic KLR algebra $\mathcal{R}_n^{\Lambda}$. This interpretation will be used to construct our graded cellular basis of $TL_{r,1,n}$ in the next section.

In similar fashion to the generalised blob algebra $B_{r,n}$, $TL_{r,1,n}$ is defined with respect to a dominant weight $\Lambda$. We will use a specific notation $TL_{r,1,n}^{\Lambda}$ to indicate which dominant weight we are using if necessary. By definition, the condition (\ref{pmres}) can be transformed into the following restriction:
\begin{equation}\label{reslmd}
    \Lambda=\Lambda_{i_1}+\Lambda_{i_2}+\dots+\Lambda_{i_r}\in P^+\textit{ where }|i_a-i_b|>2\textit{ for }a\neq b.
\end{equation}

When $n=3$, $H(r,1,3)$ is semisimple by Theorem \ref{ssth}. The following lemma gives a realisation of $TL_{r,1,3}$ as a quotient of the cyclotomic KLR algebra $\mathcal{R}_3^{\Lambda}$, which is the cornerstone of the interpretation of $TL_{r,1,n}$ into the language of KLR generators.
\begin{lem}
    Let $TL_{r,1,3}$ be the generalised Temperley-Lieb algebra defined in Definition \ref{TL} and $\mathcal{R}_3^{\Lambda}$ be the cyclotomic KLR algebra isomorphic to $H(r,1,3)$ indicated in Theorem $\ref{iso thm}$. Then
    \begin{equation}\label{tlr13}
        TL_{r,1,3}\cong \mathcal{R}_3^{\Lambda}/\mathcal{J}_{3}
    \end{equation}
    where
\begin{equation}
	\begin{aligned}
		\mathcal{J}_{3}=&\sum_{j\in I,(\alpha_i,\Lambda)>0}\langle  e(i,i+1,j) \rangle_{\mathcal{R}_{3}^{\Lambda}}\\
&+\sum_{(\alpha_{i_j},\Lambda)>0\text{ for }j=1,2,3}\langle  e(i_1,i_2,i_3) \rangle_{\mathcal{R}_{3}^{\Lambda}}.
	\end{aligned}	
\end{equation}
\end{lem}
\begin{proof}
    This can be proved by direct calculation via (\ref{isoinv}) in the isomorphism by Brundan and Kleshchev. We give an alternative proof using the cellular structure of the generalised blob algebra $B_{r,3}$. By Theorem \ref{thm 12}, we have
    \begin{equation}
        B_{r,3}\cong \mathcal{R}_3^{\Lambda}/\sum_{j\in I,(\alpha_i,\Lambda)>0}\langle  e(i,i+1,j) \rangle_{\mathcal{R}_{3}^{\Lambda}}.
    \end{equation}
    Comparing with (\ref{tlr13}), it is enough to prove
    \begin{equation}
        TL_{r,1,3}\cong B_{r,3}/\sum_{(\alpha_{i_j},\Lambda)>0\text{ for }j=1,2,3}\langle  e(i_1,i_2,i_3) \rangle_{B_{r,3}}.
    \end{equation}
    Let $\mathfrak{B}_3^{(r)}$ be the poset of one-column $r$-partitions of $3$ and $\mathfrak{B}_3^{(r)}(i),i=1,2,3$ be the subset consisting of the multipartitions with exactly $i$ non-empty partitions. For $\lambda\neq \mu\in \mathfrak{B}_3^{(r)}$ and $t\in Std(\lambda),s\in Std(\mu)$, the restriction that $\Lambda=\Lambda_{i_1}+\Lambda_{i_2}+\dots+\Lambda_{i_r}\in P^+$ where $|i_a-i_b|>2$ for $a\neq b$ implies $e(t)\neq e(s)$. As $e(i)$'s are orthogonal idempotents, for $C_{t',t}^{\lambda}$ and $C_{s,s'}^{\mu}$ in the cellular basis in Definition \ref{Df28}, we have
    \begin{equation}
        C_{t',t}^{\lambda}C_{s,s'}^{\mu}=0,
    \end{equation}
    which implies
    \begin{equation}\label{blockdecomr13}
        B_{r,3}=\bigoplus_{\lambda\in \mathfrak{B}_3^{(r)}}\langle  e_{\lambda} \rangle_{B_{r,3}}
    \end{equation}
    and each $\langle  e_{\lambda} \rangle_{B_{r,3}}$ is an indecomposable two-sided ideal of $B_{r,3}$.
    On the other hand, we have
    \begin{equation}
        TL_{r,1,3}\cong B_{r,3}/\langle  E_1 \rangle_{B_{r,3}}
    \end{equation}
    where $E_1$ is the idempotent above. The decomposition in (\ref{blockdecomr13}) implies
    \begin{equation}
        \langle  E_1 \rangle_{B_{r,3}}=\bigoplus_{\lambda\in \mathfrak{D}}\langle  e_{\lambda} \rangle_{B_{r,3}},
    \end{equation}
    where $\mathfrak{D}$ is a subset of $\mathfrak{B}_3^{(r)}$. By Theorem \ref{AK} and a simple calculation, observe that 
    \begin{equation}
        E_1\mathbb{V}_{\lambda}=0;E_0^{(i)}\mathbb{V}_{\lambda}=0 \textit{ for all } i=1,2 \textit{ and }\lambda\in \mathfrak{B}_3^{(r)}(j),j=1,2.
    \end{equation}
    Therefore, for any $\lambda\in \mathfrak{B}_3^{(r)}(j),j=1,2$, $\mathbb{V}_{\lambda}$ is a simple module of $TL_{r,1,3}$. So we have
    \begin{equation*}
        \begin{aligned}
            dim(TL_{r,1,3}) &\geq \sum_{\lambda\in \mathfrak{B}_3^{(r)}(1)} (dim(\mathbb{V}_{\lambda}))^2+\sum_{\lambda\in \mathfrak{B}_3^{(r)}(2)} (dim(\mathbb{V}_{\lambda}))^2\\
            &=3\times 1+6\times 3^2\\
            &=57.
        \end{aligned}
    \end{equation*}
    On the other hand, for any $\lambda\in \mathfrak{B}_3^{(r)}(3)$, denote by $k_1$,$k_2$ and $k_3$ the indices of the three non-empty components of $\lambda$. Let $t_{a_1a_2a_3}$ be the standard tableau of shape $\lambda$ with the number $a_i$ in the $k_i^{th}$ position.  Then by Theorem \ref{AK}, $\{t_{a_1a_2a_3}\}$ forms a basis of the irreducible module $\mathbb{V}_{\lambda}$, and we have
    \begin{equation}
    \begin{aligned}
        E_1 t_{123}=& (\frac{qv_{a_2}-v_{a_1}}{v_{a_2}-v_{a_1}}+\frac{(q-1)v_{a_3}(qv_{a_2}-v_{a_1})}{(v_{a_2}-v_{a_1})(v_{a_3}-v_{a_2})}+\frac{(q-1)v_{a_3}(v_{a_2}-qv_{a_1})}{(v_{a_2}-v_{a_1})(v_{a_3}-v_{a_1})})(t_{123}+t_{213})\\
        &+\frac{(qv_{a_3}-v_{a_1})(qv_{a_3}-v_{a_2})}{(v_{a_3}-v_{a_1})(v_{a_3}-v_{a_2})}(t_{132}+t_{231}+t_{312}+t_{321}).
    \end{aligned}
    \end{equation}
    Since $\frac{v_i}{v_j}\neq q$ for any $i\neq j$, $(qv_{a_3}-v_{a_1})(qv_{a_3}-v_{a_2})\neq 0$.
    Therefore, $E_1$ acts non-trivially on $\mathbb{V}_{\lambda}$ if $\lambda\in \mathfrak{B}_3^{(r)}(3)$.  So $TL_{r,1,3}$ has no simple modules of rank $6$.
    Therefore, $\mathfrak{B}_3^{(r)}(3)\subseteq \mathfrak{D} $. Comparing dimensions, we have $\mathfrak{B}_3^{(r)}(3)=\mathfrak{D} $ which implies
    \begin{equation}
        TL_{r,1,3}=\bigoplus_{\lambda\in \mathfrak{B}_3^{(r)}(i),i=1,2}\langle  e_{\lambda} \rangle_{B_{r,3}}.
    \end{equation}
    Therefore, 
    \begin{equation*}
         \begin{aligned}
        TL_{r,1,3}&=B_{r,3}/\bigoplus_{\lambda\in \mathfrak{B}_3^{(r)}(3)}\langle  e_{\lambda} \rangle_{B_{r,3}}\\
        &=B_{r,3}/\sum_{(\alpha_{i_j},\Lambda)>0\text{ for }j=1,2,3,i_j\neq i_k \textit{ for }j\neq k}\langle  e(i_1,i_2,i_3) \rangle_{B_{r,3}}\\
        &=\mathcal{R}_3^{\Lambda}/\mathcal{J}_{3}.
    \end{aligned}   
    \end{equation*}

\end{proof}
For any $n\geq 3$, $\mathcal{R}_{3}^{\Lambda}$ can be regarded as a subalgebra of $\mathcal{R}_{n}^{\Lambda}$. Denote by $\mathcal{J}_{n}$ the two-sided ideal of $\mathcal{R}_{n}^{\Lambda}$ which is generated by $\mathcal{J}_{3}$. Lemma $(\ref{keylm})$ implies
\begin{equation}\label{ieq}
	\begin{aligned}
		\mathcal{J}_{n}(\Lambda)=&\sum_{\mathbf{i}\in I^{n-2},(\alpha_i,\Lambda)>0}\langle  e(i,i+1,\mathbf{i}) \rangle_{\mathcal{R}_{n}^{\Lambda}}\\
&+\sum_{\mathbf{i}\in I^{n-3},(\alpha_{i_j},\Lambda)>0\text{ for }j=1,2,3,i_j\neq i_k \textit{ for }j\neq k}\langle  e(i_1,i_2,i_3,\mathbf{i}) \rangle_{\mathcal{R}_{n}^{\Lambda}}.
	\end{aligned}	
\end{equation}

It now follows that the generalised Temperley-Lieb algebra $TL_{r,1,n}$ can be realised as a quotient of the KLR algebra, that is
\begin{thm} $\label{dftl}$
    Let $TL_{r,1,n}$ be the generalised Temperley-Lieb algebra in Definition \ref{TL} and $\mathcal{R}_n^{\Lambda}$ be the cyclotomic KLR algebra isomorphic to the corresponding Hecke algebra $H(r,1,n)$ with $\Lambda$ satisfying $(\ref{reslmd})$.
	 Then
	\begin{equation}
		TL_{r,1,n}\cong  \mathcal{R}_n^{\Lambda}/\mathcal{J}_{n},
	\end{equation} 
	where $\mathcal{J}_{n}$ is the two-sided ideal of $\mathcal{R}_n^{\Lambda}$ in (\ref{ieq}).
\end{thm}
In \cite{LOBOSMATURANA2020106277}, Lobos and  Ryom-Hansen introduce a KLR interpretation of the generalised blob algebra $B_{r,n}$ as a quotient of $\mathcal{R}_n^{\Lambda}$:

\begin{thm}(\cite{LOBOSMATURANA2020106277}, Theorem 42)  $\label{thm 12} $
	Let $B_{r,n}$ and $\mathcal{R}_n^{\Lambda}$ be as defined above, then $B_{r,n}=\mathcal{R}_n^{\Lambda}/\mathcal{J}_{1,n}$ where
	\begin{equation}
		\mathcal{J}_{1,n}=\sum_{\mathbf{i}\in I^{n-2},(\alpha_i,\Lambda)>0}\langle  e(i,i+1,\mathbf{i}) \rangle_{\mathcal{H}_{n}^{\Lambda}}.
	\end{equation}  
\end{thm}

Since $\mathcal{J}_{1,n}$ is contained in the ideal $\mathcal{J}_{n}$, the generalised Temperley-Lieb algebra $TL_{r,1,n}$ is realised as a quotient of the generalised blob algebra $B_{r,n}$:
\begin{cor}
    $\label{blockdecom}$
	The generalised Temperley-Lieb algebra $TL_{r,1,n}$ is a quotient of the generalised blob algebra $B_{r,n}$ by the two-sided ideal
\begin{equation*}
    \mathcal{J}_{2,n}=	\sum_{\mathbf{i}\in I^{n-3},(\alpha_{i_j},\Lambda)>0\text{ for }j=1,2,3,i_j\neq i_k \textit{ for }j\neq k}\langle  e(i_1,i_2,i_3,\mathbf{i}) \rangle_{B_{r,n}}.
\end{equation*}
\end{cor}

\section{The graded cellularity of $TL_{r,1,n}$}$\label{gcs}$
	In this section, we show that $TL_{r,1,n}$ is a graded cellular algebra when $r\geq 2$. The cellularity of $TL_{1,1,n}$ is explained in Example 1.4 in \cite{Lehrer1996}. We construct a cellular structure of $TL_{r,1,n}$ using a truncation of that of the generalised blob algebra $B_{r,n}$ given by Lobos and Ryom-Hansen in \cite{LOBOSMATURANA2020106277}.
	
	 We begin by recalling the definition of cellular algebras.
	\begin{defn}$\label{dfc}$
(\cite{Lehrer1996}, Definition 1.1)
	Let $R$ be a commutative ring with unit and $A$ be an $R$-algebra. Then $A$ is called cellular if it has a cell datum $(\Lambda,i,M,C)$ consisting of
	
	a) A finite poset $\Lambda$;
	
	b) An $R$-linear anti-involution $i$;
	
	c) A finite, non-empty set $M(\lambda)$ of indices for every $\lambda \in \Lambda$;
	
	d) An injection 
	\begin{equation*}
	\begin{aligned}
		C: \cup_{\lambda \in \Lambda}M(\lambda)\times M(\lambda) & \to A \\
		(S,T) &\mapsto C_{S,T}^{\lambda}
	\end{aligned}	    
	\end{equation*}

	satisfying the following conditions:
	
	1) The image of $C$ forms an $R$-basis of $A$.
	
	2) $i(C_{S,T}^{\lambda})=C_{T,S}^{\lambda}$ for all elements in the basis.
	
	3) Let $A_{<\lambda}$ be the $R$-span of all the elements of form $C_{X,Y}^{\mu}$ with $\mu <\lambda$ in the poset.
	Then for all $a\in A$,
	\begin{equation}\label{55}
		aC_{S,T}^{\lambda}=\sum_{S'\in M(\lambda)} r_a(S',S)C_{S',T}^{\lambda}  \text{    } mod \text{   }A_{<\lambda}
	\end{equation}
	with the coefficients $r_a(S',S)$ independent of $T$.
\end{defn}
	
	Let $\Lambda'$ be a downward closed subset of $\Lambda$. Then  $A(\Lambda'):=\langle C_{S,T}^{\lambda}|\lambda\in\Lambda'\rangle_R$ forms a two-sided ideal of $A$. Further, if $\Lambda_1\subseteq\Lambda_2$ are two downward closed subsets of $\Lambda$, then $A(\Lambda_1)\subseteq A(\Lambda_2$ and $A(\Lambda_2)/A(\Lambda_1)$ can be regarded as an $(A,A)$-bimodule. In particular, for $\lambda\in\Lambda$, denote by $A(\{\lambda \})$ the bimodule $A(\Lambda_2)/A(\Lambda_1)$ where $\Lambda_2=\{\mu\in\Lambda|\mu\leq\lambda \}$ and $\Lambda_1=\{\mu\in\Lambda|\mu<\lambda \}$.

Let us recall the definition of a cell module which is introduced by Graham and Lehrer in \cite{Lehrer1996}.
\begin{defn}
(\cite{Lehrer1996}, Definition 2.1)
	For each $\lambda\in \Lambda$ define the (left) $A$-module $W(\lambda)$ as follows:
	
	$W(\lambda)$ is a free module with basis $\{C_S|S\in M(\lambda) \}$ and the $A$-action is defined by
	\begin{equation}
		aC_S=\sum_{S'\in Std(\lambda)}r_a(S',S)C_{S'}
	\end{equation}
    for all $a\in A$ and $S\in M(\lambda)$ and $r_a(S',S)$ is the coefficient in $(\ref{55})$. The module $W(\lambda)$ is called the (left) cell module of $A$ corresponding to $\lambda$.
\end{defn}

Similarly, we can define a right $A$-module $W(\lambda)^*$. There is a natural isomorphism of $R$-modules $C^{\lambda}:W(\lambda)\otimes W(\lambda)^*\mapsto A(\{\lambda \})$ defined by $(C_S,C_T)\mapsto C_{S,T}^{\lambda}$.

There is a bilinear form $\phi_{\lambda}(\;,\;)$ on $ W(\lambda ) $ which is defined as follows and extended bilinearly:
\begin{equation}\label{81}
			C_{S',T}^{\lambda}C_{S,T'}^{\lambda}=\phi_{\lambda}(C_T,C_S)C_{S',T'}^{\lambda}  \text{    } mod \text{   }A_{<\lambda}.
\end{equation}

For $\lambda\in \Lambda$, define $rad(\phi_{\lambda}):=\{x\in W(\lambda)|\phi_{\lambda}(x,y)=0 \text{ for all } y\in W(\lambda)\}$. Let $\Lambda_0=\{\lambda\in \Lambda|\phi_{\lambda}\neq 0 \}$. Graham and Lehrer list the simple modules of $A$ in \cite{Lehrer1996}:

\begin{thm} $\label{33}$
	(\cite{Lehrer1996}, Proposition 3.2,Theorem 3.4)
	Let $A$ be a cellular algebra with the cell datum $(\Lambda,i,M,C)$ over a base ring $R$ which is a field. For $\lambda\in \Lambda$, let $W(\lambda)$ be the cell module of $A$ and $rad(\phi_{\lambda})$ be as defined above. Then we have
	
	(i) $rad(\phi_{\lambda})$ is an $A$-submodule of $W(\lambda)$;
	
	(ii) If $\phi_{\lambda}\neq 0$, the quotient $W(\lambda)/rad(\phi_{\lambda})$ is absolutely irreducible;
	
	(iii) The set $\{L(\lambda)=W(\lambda)/rad(\phi_{\lambda})|\lambda\in\Lambda_0 \}$ is a complete set of absolutely irreducible $A$-modules.
\end{thm}

In \cite{hu2010graded}, Hu and Mathas introduce a graded cell datum as a development of the original one, which is equipped with a degree function 
\begin{equation*}
	    deg:  \cup_{\lambda \in \Lambda}M(\lambda)\mapsto \mathbb{Z},
	\end{equation*}
	which satisfies $C_{S,T}^{\lambda}$ is homogeneous and of degree $deg(C_{S,T}^{\lambda})=deg(S)+deg(T)$ for all $\lambda\in\lambda$ and $S,T\in M(\lambda)$.

 We next recall the graded cellular basis of $B_{r,n}$ constructed by Lobos and Ryom-Hansen in \cite{LOBOSMATURANA2020106277}. Let $\mathfrak{B}_n^{(r)}$ be the poset of one-column multipartitions of $n$ defined in \ref{1cmltp}. For $\lambda\in\mathfrak{P}_n^{(r)}$, let $t^{\lambda}$ be the unique tableau of shape $\lambda$ such that $t^{\lambda}(i)<t^{\lambda}(j)$ if $1\leq i<j\leq n$. For each $t\in Std(\lambda)$, denote by $d(t)\in \mathfrak{S}_n$ the permutation such that $t=t^{\lambda}\circ d(t)$ where $\circ$ is the natural $\mathfrak{S}_n$-action on the tableaux. For a multipartition $\lambda\in \mathfrak{P}_n^{(r)}$ and a node $\gamma=(a,1,l)\in [\lambda]$, define the residue of $\gamma$ as
\begin{equation}\label{eq75}
	Res^{\Lambda}(\gamma):=1-a+i_l
\end{equation}
where $i_l$ is the subscript of the dominate weight $\Lambda=\Lambda_{i_1}+\Lambda_{i_2}+\dots+\Lambda_{i_r}$.
For $t\in Tab(\lambda)$ and $1\leq k\leq n$, set $Res_t^{\Lambda}(k)=Res^{\Lambda}(\gamma)$, where $\gamma$ is the unique node such that $t(\gamma)=k$. Define the residue sequence of $t$ as follows:
\begin{equation} \label{eqres}
	res^{\Lambda}(t):=(Res_t^{\Lambda}(1),Res_t^{\Lambda}(2),\dots,Res_t^{\Lambda}(n))\in I^n.
\end{equation}
For example, let $\Lambda=\Lambda_0+\Lambda_3+\Lambda_7$ and $t$ be the following tableau, then $res^{\Lambda}(t)=(0,3,2,-1,-2,1,-3,-4,-5)$.
\tikzset{every picture/.style={line width=0.75pt}} 
\begin{center}

\begin{tikzpicture}[x=0.75pt,y=0.75pt,yscale=-1,xscale=1]

\draw  [draw opacity=0] (480,43) -- (509.94,43) -- (509.94,222.16) -- (480,222.16) -- cycle ; \draw   (480,43) -- (480,222.16)(509.86,43) -- (509.86,222.16) ; \draw   (480,43) -- (509.94,43)(480,72.86) -- (509.94,72.86)(480,102.71) -- (509.94,102.71)(480,132.57) -- (509.94,132.57)(480,162.42) -- (509.94,162.42)(480,192.28) -- (509.94,192.28)(480,222.13) -- (509.94,222.13) ; \draw    ;
\draw  [draw opacity=0] (567.34,43) -- (598.15,43) -- (598.15,132.67) -- (567.34,132.67) -- cycle ; \draw   (567.34,43) -- (567.34,132.67)(597.19,43) -- (597.19,132.67) ; \draw   (567.34,43) -- (598.15,43)(567.34,72.86) -- (598.15,72.86)(567.34,102.71) -- (598.15,102.71)(567.34,132.57) -- (598.15,132.57) ; \draw    ;

\draw (484.58,51.67) node [anchor=north west][inner sep=0.75pt]   [align=left] {1};
\draw (484.58,81.52) node [anchor=north west][inner sep=0.75pt]   [align=left] {4};
\draw (484.58,111.38) node [anchor=north west][inner sep=0.75pt]   [align=left] {5};
\draw (484.58,141.23) node [anchor=north west][inner sep=0.75pt]   [align=left] {7};
\draw (484.58,171.09) node [anchor=north west][inner sep=0.75pt]   [align=left] {8};
\draw (484.58,200.95) node [anchor=north west][inner sep=0.75pt]   [align=left] {9};
\draw (571.92,51.67) node [anchor=north west][inner sep=0.75pt]   [align=left] {2};
\draw (571.92,81.52) node [anchor=north west][inner sep=0.75pt]   [align=left] {3};
\draw (571.92,111.38) node [anchor=north west][inner sep=0.75pt]   [align=left] {6};
\draw (649.43,50.09) node [anchor=north west][inner sep=0.75pt]    {$\emptyset $};

\end{tikzpicture}
\end{center}
\begin{defn} $\label{dfei}$
	Suppose that $\lambda\in \mathfrak{B}_n^{(r)}$ and $t\in Tab(\lambda)$ is a tableau of shape $\lambda$. Let $t^{(\lambda)}$ be the standard tableau of shape $\lambda$ defined above. Set
\begin{equation*}
	\begin{aligned}
	    e(t)&=e(res^{\Lambda}(t));\\
		e_{\lambda}&:=e(res^{\Lambda}(t^{(\lambda)})),
	\end{aligned}    
\end{equation*}	

	where $e(i)$ is the idempotent corresponding to the eigenspace $M_i$ in (\ref{es}).
\end{defn}
 Let $*$ be the unique $R$-linear anti-automorphism of the KLR algebra $\mathcal{R}_n^{\Lambda}$ introduced by Brundan and Kleshchev in section 4.5 of \cite{BrundanKleshchev_2009} which fixes each of the generators in Definition $\ref{klrdf}$.

\begin{defn} $\label{Df28}$
		Suppose that $\lambda \in \mathfrak{B}_n^{(r)}$ and $s,t\in Std(\lambda)$ and fix reduced expressions $d(s)=s_{i_1}s_{i_2}\dots s_{i_k}$ and $d(t)=s_{j_1}s_{j_2}\dots s_{j_m}$. Define
		\begin{equation} \label{76}
		C_{s,t}^{\lambda}=\psi_{d(s)}^* e_{\lambda}\psi_{d(t)},
	\end{equation}
	where $\psi_{d(s)}=\psi_{i_1}\psi_{i_2}\dots \psi_{i_k}$ and $\psi_{d(t)}=\psi_{j_1}\psi_{j_2}\dots \psi_{j_m}$.
\end{defn}

Define the degree function $deg: \bigsqcup_{\lambda\in \mathfrak{B}_n^{(r)}} Std(\lambda)\mapsto \mathbb{Z}$ as
\begin{equation} \label{degt}
    deg(t):=deg(\psi_{d(t)}^* e_{\lambda}),
\end{equation}
 where the degree function on the right hand side is the one defined in $(\ref{degf})$. We next recall the graded cell datum of $\mathfrak{B}_n^{(r)}$ introduced by Lobos and Ryom-Hansen:

\begin{thm}$\label{29}$
	(\cite{LOBOSMATURANA2020106277}, Theorem 38)
	Let $\mathfrak{B}_n^{(r)}$ be the poset of one-column $r$-partitions of $n$ and $C,deg$ be as defined above.
	$(\mathfrak{B}_n^{(r)},*,Std,C,deg)$ is a graded cell datum of $B_{r,n}^{\Lambda}$. In other words, $\{C_{s,t}^{\lambda}|\lambda\in \mathfrak{B}_n^{(r)},s,t\in Std(\lambda) \}$ is a graded cellular basis of $B_{r,n}^{\Lambda}$ with respect to the partial order $\unlhd $.
\end{thm} 

Readers may notice that the KLR generators $y_k$ are not used in the cellular basis. The following lemma shows that $y_ke_{\lambda}$ can be realised as a combination of lower terms.
	
	\begin{lem} $\label{lrh lm}$
		(\cite{LOBOSMATURANA2020106277}, Lemma 17, Lemma 37)
		
		For $\lambda \in \mathfrak{B}_n^{(r)}$ and $1\leq k \leq n$, we have
		\begin{equation}
			y_ke_{\lambda}=e_{\lambda}y_k=\sum_{\mu \triangleleft \lambda} D_{\mu}
		\end{equation}
	    where the sum runs over $\mu \in \mathfrak{B}_n^{(r)}$ and $D_{\mu}$ is in the two-sided ideal generated by $e_{\mu}$.
	    \end{lem}
 
 We now concentrate on a truncation of this cellular structure. Let $\mathfrak{D}_n^{(r)}:=\coprod_{k\geq 3} \mathfrak{B}_n^{(r)}(k)$ be the subset of $\mathfrak{B}_n^{(r)}$ consisting of multipartitions with at least 3 non-empty components. By Lemma $\ref{dw}$, $\mathfrak{D}_n^{(r)}$ is downward closed with respect to the partial order $\unlhd $ defined by (\ref{podf}).
 \begin{lem} \label{triquo}
   Denote by $\mathcal{I}_{\mathfrak{D}}$ the two-sided ideal of $B_{r,n}$ generated by  $\{C_{s,t}^{\lambda}\}$ where $\lambda$ runs over $\mathfrak{D}_n^{(r)}$. Let $\mathcal{J}_{2,n}$ be the two-sided ideal in Corollary $\ref{blockdecom}$. Then we have	$\mathcal{I}_{\mathfrak{D}}=\mathcal{J}_{2,n}$.
 \end{lem}
\begin{proof}
	By the definition of $t^{\lambda}$, for any $\lambda \in\mathfrak{D}_n^{(r)}$, $e(t^{\lambda})$ is of the form $e(i_1,i_2,i_3,\mathbf{i})$ where $(\alpha_{i_j},\Lambda)>0$ for $j=1,2,3$ and $\mathbf{i}\in I^{n-3}$. Therefore, $C_{s,t}^{\lambda}\in \mathcal{J}_{2,n}$. So $\mathcal{I}_{\mathfrak{D}}$ is contained in $\mathcal{J}_{2,n}$.
	
	We next show that $e(i_1,i_2,i_3,\mathbf{i})\in \mathcal{I}_{\mathfrak{D}}$ where $(\alpha_{i_j},\Lambda)>0$ for $j=1,2,3$ and $\mathbf{i}\in I^{n-3}$. 
	We prove this by showing that the image of $e(i_1,i_2,i_3,\mathbf{i})$ in the truncation $B_{r,n}/\mathcal{I}_{\mathfrak{D}}$ is trivial.
	
	Let $\lambda\in \coprod_{k= 1,2} \mathfrak{B}_n^{(r)}(k)$, $s,t\in Std(\lambda)$ and $C_{s,t}^{\lambda}$ is the element defined in the cellular basis in Definition $\ref{Df28}$. 
	
	If $\lambda\in \mathfrak{B}_n^{(r)}(1)$, we have $s=t=t^{\lambda}$ and $C_{s,t}^{\lambda}=e_{\lambda}=e(j,j-1,j-2,\dots,j+1-n)$ for some $j\in I$. The condition $(\ref{pmres})$ implies that $|i_1-i_2|\neq 1$. Therefore, $e(i_1,i_2,i_3,\mathbf{i})\neq e_{\lambda}$ which implies
	\begin{equation}\label{orth1}
	    e(i_1,i_2,i_3,\mathbf{i}) e_{\lambda}=e(i_1,i_2,i_3,\mathbf{i})C_{s,t}^{\lambda}=0.
	\end{equation}
	
	If $\lambda\in \mathfrak{B}_n^{(r)}(2)$, we have 
	\begin{equation*}
	    C_{s,t}^{\lambda}=\psi_{d(s)}^* e_{\lambda}\psi_{d(t)}=e(s)\psi_{d(s)}^*\psi_{d(t)}.
	\end{equation*}
	Let $e(j_1,j_2,j_3,\dots,j_n)=e(s)$. As $s\in Std(\lambda)$ and $\lambda\in \mathfrak{B}_n^{(r)}(2)$, two of $1,2$ and $3$ must be in the same column of $s$ which implies $j_1=j_2+1$, $j_2=j_3+1$ or $j_1=j_3+1$. Comparing with the restriction on $i_1$, $i_2$ and $i_3$, we have $e(i_1,i_2,i_3,\mathbf{i})\neq e(s)$. Therefore,
	\begin{equation}\label{orth2}
	    e(i_1,i_2,i_3,\mathbf{i}) C_{s,t}^{\lambda} =0.
	\end{equation}
	
	As $B_{r,n}/\mathcal{I}_{\mathfrak{D}}$ is a truncation of $B_{r,n}$ with respect to the cellular structure in Theorem $\ref{29}$ and the ideal $\mathfrak{D}_n^{(r)}:=\coprod_{k\geq 3} \mathfrak{B}_n^{(r)}(k)$, $\{C_{s,t}^{\lambda}| \lambda\in  \coprod_{k= 1,2} \mathfrak{B}_n^{(r)}(k)\}$ is a cellular basis of it. As $B_{r,n}$ is an algebra with $1$, $(\ref{orth1})$ and $(\ref{orth2})$ imply that the image of $e(i_1,i_2,i_3,\mathbf{i})$ in this truncation is $0$, so $e(i_1,i_2,i_3,\mathbf{i})\in \mathcal{I}_{\mathfrak{D}}$. Since the ideal $\mathcal{J}_{2,n}$ is generated by $e(i_1,i_2,i_3,\mathbf{i})$, we have $\mathcal{J}_{2,n}\subseteq\mathcal{I}_{\mathfrak{D}}$.

Therefore, $\mathcal{J}_{2,n}=\mathcal{I}_{\mathfrak{D}}$.
\end{proof}

As an immediate consequence of this lemma and Theorem $\ref{29}$, a cellular structure of $TL_{r,1,n}$ can be obtained by a truncation of that of the generalised blob algebra $B_{r,n}$. In particular,
\begin{thm} $\label{cb33}$
  Let $\mathfrak{B}_n^{(r)}$ be the set of one-column $r$-partitions of $n$ and $\mathfrak{D}_n^{(r)}$ be the subset consisting multipartitions with more than two non-empty components.
  For $\lambda\in \mathfrak{B}_n^{(r)}-\mathfrak{D}_n^{(r)}$ and $s,t\in Std(\lambda)$, let $C_{s,t}^{\lambda}$ be the element defined in $(\ref{76})$ and $deg$ be the degree function on the tableaux defined in $(\ref{degt})$. Let $TL_{r,1,n}$ be the generalised Temperley-Lieb algebra over a field $R$ of characteristic $0$ defined in Definition \ref{TL} (cf. Theorem \ref{dftl}) and $*$ be the anti-automorphism of $TL_{r,1,n}$ fixing the KLR generators.
	Then $TL_{r,1,n}$ is a graded cellular algebra with the cell datum $(\mathfrak{B}_n^{(r)}-\mathfrak{D}_n^{(r)},*,Std,C,deg)$ with respect to the partial order $\unlhd $ defined by (\ref{podf}). Specifically, for any $a\in TL_{r,1,n}$, $\lambda\in \mathfrak{B}_n^{(r)}-\mathfrak{D}_n^{(r)}$ and $s,t\in Std (\lambda)$, we have
	\begin{equation}
		aC_{s,t}^{\lambda}=\sum_{s'\in Std(\lambda)}r_a(s',s)C_{s',t}^{\lambda}+\sum_{\mu \triangleleft \lambda,u,v\in Std(\mu)}c_a(s,t,u,v)C_{u,v}^{\mu}
	\end{equation}
    where $r_a(s',s)\in R$ does not depend on $t$ and $c_a(s,t,u,v)\in R$.
\end{thm}

Notice that the cellular bases of $B_{r,n}^{\Lambda}$ and $TL_{r,1,n}^{\Lambda}$ are heavily dependent on the fixed reduced expression for each $d(t)$. In order to facilitate calculations, we fix a special set of reduced expressions, called ``official'', for the $d(t)$, which satisfy the following property. 

\begin{defn}
	Let $\lambda\in  \mathfrak{B}_n^{(r)}$ be a one-column multipartition and for $t\in Std(\lambda)$,  denote by $d(t)$ the permutation such that $t=t^{\lambda}\circ d(t)$. A set of fixed reduced expressions $ORE(\lambda):=\{d(t)=s_{j_1}s_{j_2}\dots s_{j_m}|t\in Std(\lambda) \}$ is called official if for each $t\in Std(\lambda)$, there is a unique reduced expression $d(t)=s_{j_1}s_{j_2}\dots s_{j_m}\in ORE(\lambda)$ and $s_{j_1}s_{j_2}\dots s_{j_{m-1}}\in ORE(\lambda)$ is a reduced expression of $d(s)$ for some $s\in Std(\lambda)$.
\end{defn}

In order to prove the existence of an official set of reduced expressions, we introduce a partial order on $Tab(\lambda)$, the set of tableaux of shape $\lambda$. For $t\in Tab(\lambda)$, denote by $t|_k$ the tableau obtained by deleting $k+1,k+2,\dots,n$ from $t$.
 For example, if $t$ is the tableau in Figure-$\ref{oct}$, then $t|_5$ is as in Figure-$\ref{rst}$.
 \begin{figure}[h]
	\begin{center}
		\begin{tikzpicture}[scale=1]
			\draw (0,1)--(0,4);\draw (1,1)--(1,4);\draw (4,2)--(4,4);\draw (5,2)--(5,4);
			\foreach \x in {1,2,3,4}
			\draw (0,\x)--(1,\x);
			\foreach \x in {2,3,4}
			\draw (4,\x)--(5,\x);

			\draw node at (0.5,3.5){1};\draw node at (4.5,3.5){2};\draw node at (0.5,1.5){5};
			\draw node at (4.5,2.5){3};\draw node at (0.5,2.5){4};
		    \draw node at (8.5,3.5){$\emptyset$};
			
		\end{tikzpicture}
	
	\end{center}
	\caption{A restriction to the tableau $t$ } \label{rst}
\end{figure}
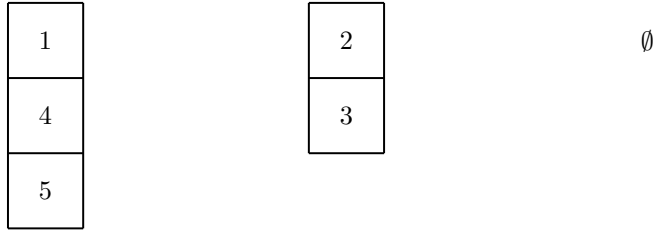
 
 For $s,t\in Tab(\lambda)$, we say $s\unlhd t$ if
\begin{equation} \label{potab}
	shape(s|_k)  \unlhd shape(t|_k) 
\end{equation}
for all $1\leq k\leq n$, where $\unlhd$ is the partial order defined by (\ref{podf}). Note that $shape(s|_k)$ is not necessarily a multipartition since $s$ can be non-standard but we can still define a partial order on the set consisting $shape(s|_k)$ for all $s\in Tab(\lambda)$ using (\ref{podf}). Let $\leq $ be the Bruhat order on $\mathfrak{S}_n$. The following Theorem builds a connection between the posets of tableaux and the elements in the symmetric group.
\begin{thm}
    (\cite{LOBOSMATURANA2020106277}, Theorem 4)	Let $\lambda\in  \mathfrak{B}_n^{(r)}$ be a one-column multipartition and $s,t\in Tab(\lambda)$. Then $d(s)\leq d(t)$ if and only if $s\unlhd t$.
\end{thm}

We next show how to construct an official set of reduced expressions $ORE(\lambda)$ for each $\lambda\in  \mathfrak{B}_n^{(r)}$. We do this progressively in the set $Std(\lambda)$, starting with the tableau with a trivial permutation. Denote by $S$ the set of tableaux corresponding to which we have chosen reduced expressions. We start with $S=\{t^{\lambda} \}$ and the corresponding set of reduced expressions is $ore=\{d(t^{\lambda})=1 \}$. 

If $S=Std(\lambda)$ then let $ORE(\lambda)=ore$ and $ORE(\lambda)$ is an official set of reduced expressions. 

Otherwise, choose $t\in Std(\lambda)-S$ such that $l(d(t))=m$ is the smallest among permutations corresponding to the left standard tableaux. As $t^{\lambda}\in S$, $t\neq t^{\lambda}$. Let $1\leq i\leq n$ be the smallest number such that $t(i)>t(i+1)$.
Denote $s=t\circ s_i$. Then by definition, $s\triangleleft t$. So $d(s)<d(t)$. On the other hand, $t(i)>t(i+1)$ implies that $i$ is not above $i+1$, so $s\in Std(\lambda)$. Since $l(d(t))$ is the smallest, we have $s\in S$. Thus there exists $d(s)=s_{j_1}s_{j_2}\dots s_{j_{m-1}}\in ore$. Add $t$ to $S$ as well as $d(t)=s_{j_1}s_{j_2}\dots s_{j_{m}}$ to $ore$ where $j_m=i$ defined above. Then check whether $S=Std(\lambda)$. If not, repeat this procedure.

As there are finitely many elements in $Std(\lambda)$, the algorithm above confirms the existence of an official set of reduced expressions. 

It should be remarked that the official set is not unique. As a counterexample, let $\lambda=((1^2),(1^2))$. There are 6 standard tableaux of shape $\lambda$. Denote by $t_{ij}$ the tableau with $i$ and $j$ in the first component. There are two official sets of reduced expressions:
\begin{align*}
    \textit{Standard tableaux}:& &t_{13}&&t_{12}&&t_{14}&&t_{23}&&t_{24}&&t_{34}\\
   \textit{ Official set 1}:&&1&&s_2&&s_3&&s_1&&s_1s_3&&s_1s_3s_2\\
   \textit{ Official set 2}:&&1&&s_2&&s_3&&s_1&&s_3s_1&&s_3s_1s_2.
\end{align*}

In the following sections, we fix one official set of reduced expressions $ORE(\lambda)$ for all $\lambda\in \mathfrak{B}_n^{(r)}$. Observe that the `official' property is one of the whole set $\{d(t)=s_{j_1}s_{j_2}\dots s_{j_m}|t\in Std(\lambda) \}$ rather than of a single permutation corresponding to some tableau. In the following sections, we always assume that the chosen set of reduced expressions for the cellular basis is official. To make the statement briefer, we also call these reduced expressions the official ones.

\section{The cell modules and semisimplicity of $TL_{r,1,n}$}

In this section, we first calculate the dimensions of the cell modules of $TL_{r,1,n}$ when $r\geq 2$ and then give several equivalent conditions for the semisimplicity of $TL_{r,1,n}$.
Let $\lambda\in \mathfrak{B}_n^{(r)}-\mathfrak{D}_n^{(r)}$ be a one-column multipartition of $n$ with at most two non-empty components (cf. Theorem $\ref{cb33}$) and $W(\lambda)$ be the corresponding cell module of $TL_{r,1,n}$ with respect to the cell datum in Theorem \ref{cb33}. 
Denote by $a_1$ the number of nodes in the first non-empty component of $\lambda$. Then the number of standard tableaux of shape $\lambda$ is $\binom{n}{ a_1}$. Hence we have
\begin{equation}
	dim(W(\lambda))= \binom{n}{ a_1}.
\end{equation}
Denoting by $TL(\{\lambda \})$ the corresponding $TL_{r,1,n}$-bimodule, then
\begin{equation}
	dim(TL(\{\lambda \}))=\binom{n}{ a_1}^2.
\end{equation}

Further,
\begin{equation*}
 \begin{aligned}
	dim(TL_{r,1,n})&=\sum_{\lambda\in\mathfrak{B}_n^{(r)}-\mathfrak{D}_n^{(r)}} dim(TL(\{\lambda \}))\\
	             &=\sum_{\lambda\in\mathfrak{B}_n^{(r)}(1)} dim(TL(\{\lambda \}))+\sum_{\lambda\in\mathfrak{B}_n^{(r)}(2)} dim(TL(\{\lambda \}))\\
	             &=r+\binom{r}{2}(\sum_{a_1=1}^{n-1}\binom{n}{ a_1}^2)\\
	             &=r+\binom{r}{2}(\binom{2n}{n}-2)\\
	             &=\binom{r}{2}\binom{2n}{n}-r^2+2r.
\end{aligned}   
\end{equation*}

In the case where $r=2$, we have $dim(TL_{2,1,n})=\binom{2n}{n}$, which is the dimension of $TLB_n(q,Q)$, the Temperley-Lieb algebra of type $B_n$. 
We next concentrate on the irreducible representations of $TL_{r,1,n}$. A direct consequence of the definition of the cellular basis given in $(\ref{76})$ is as follows:
\begin{lem}$\label{34}$
	Let $\phi_{\lambda}(,)$ be the bilinear form defined in $(\ref{81})$. For any $\lambda\in \mathfrak{B}_n^{(r)}-\mathfrak{D}_n^{(r)}$,  $ \phi_{\lambda}\neq 0$.
\end{lem}
\begin{proof}
	Let $t^{\lambda}$ be the standard tableau of shape $\lambda$ such that $t^{\lambda}(i)<t^{\lambda}(j)$ for all $1\leq i<j\leq n$. Then we have
	\begin{equation*}
	   	\begin{aligned}
		C_{s,t^{\lambda}}^{\lambda}&=\psi_{d(s)}^* e_{\lambda}\\
		C_{t^{\lambda},s'}^{\lambda}&=e_{\lambda}\psi_{d(s')}.
	\end{aligned} 
	\end{equation*}

    Thus,
    \begin{equation*}
      \begin{aligned}
    	C_{s,t^{\lambda}}^{\lambda}C_{t^{\lambda},s'}^{\lambda}&=\psi_{d(s)}^* e_{\lambda} e_{\lambda}\psi_{d(s')}\\
    	&=\psi_{d(s)}^*  e_{\lambda}\psi_{d(s')}\\
    	&=C_{s,s'}^{\lambda}
    \end{aligned}      
    \end{equation*}

    which implies $\phi_{\lambda}(t^{\lambda},t^{\lambda})=1$.
\end{proof}

This lemma guarantees that $TL_{r,1,n}$ is a quasi-hereditary algebra.

For two residues, $a$ and $b$, in $\mathbb{Z}/e\mathbb{Z}$, We define the difference between them as
\begin{equation}\label{eq721}
    |a-b|:=min\{|c-d|\in\mathbb{N}|c,d\in\mathbb{Z}\textit{ and }c\equiv a, d\equiv b \textit{ mod } e\} .
\end{equation}
The next theorem gives an equivalent condition for the semi-simplicity of $TL_{r,1,n}^{\Lambda}$ for some specific dominant weight $\Lambda$.

\begin{thm} $\label{Th35}$
	Let $\Lambda=\Lambda_{j_1}+\Lambda_{j_2}+\dots+\Lambda_{j_r}$ be a dominant weight satisfying $(\ref{reslmd})$ and $TL_{r,1,n}^{\Lambda}$ be the generalised Temperley-Lieb algebra defined in Theorem \ref{dftl} over a field $R$ of characteristic $0$. Then $TL_{r,1,n}^{\Lambda}$ is semi-simple if and only if $|j_k-j_l|\geq n$ for all $1\leq k<l \leq r$, where $|j_k-j_l|$ is as defined in (\ref{eq721}).
\end{thm}
Before proving this theorem, we remark that the condition $|j_k-j_l|\geq n$ is not only a restriction on the dominant weight $\Lambda$ but also on the cardinality of the index set $I$, which is the smallest positive integer $e$ such that $1+q+q^2+\dots+q^{e-1}=0$. Since $I=\mathbb{Z}/e\mathbb{Z}$ if $e>0$ and all $j_k$'s are in $I$, this condition implies that $e\geq rn$ if $e\neq 0$.
\begin{proof}
	We first check the sufficiency. Suppose $\Lambda=\Lambda_{j_1}+\Lambda_{j_2}+\dots+\Lambda_{j_r}$ satisfying that  $|j_k-j_l|\geq n$ for all $1\leq k<l \leq r$. By Theorem $\ref{33}$, it is enough to show that $rad(\phi_{\lambda})=0$ for all $\lambda \in \mathfrak{B}_n^{(r)}-\mathfrak{D}_n^{(r)}$. Without losing generality, assume that the last $r-2$ components of $\lambda$ are empty.
	
	Let $t^{\lambda}$ be the tableau of shape $\lambda$ such that $t^{\lambda}(i)<t^{\lambda}(j)$ if $1\leq i<j\leq n$. For $w\in \mathfrak{S}_n$, let $t_w^{\lambda}=t^{\lambda}\circ w$. We claim that $\phi_{\lambda}(t_w^{\lambda},t_w^{\lambda})=1$, if $t_w^{\lambda}\in Std (\lambda)$. This will be proved by induction on $l(w)$.
	
	When $l(w)=0$, $t_w^{\lambda}=t^{\lambda}$ and we have  $\phi_{\lambda}(t^{\lambda},t^{\lambda})=1$ by the proof of Lemma $\ref{34}$. For a positive integer $k$, assume $\phi_{\lambda}(t_w^{\lambda},t_w^{\lambda})=1$ for all $w\in \mathfrak{S}_n$ such that $l(w)<k$ and $t_w^{\lambda}$ is a standard tableau. For $w\in \mathfrak{S}_n$ with length $k$, let $w=s_{u_1}s_{u_2}\dots s_{u_k}$ be the official reduced expression. Then $w'=s_{u_1}s_{u_2}\dots s_{u_{k-1}}$ is the official reduced expression of $w'$ which is of length $k-1$. By the inductive assumption, we have:
	\begin{equation}
		e_{\lambda}\psi_{w'}\psi_{w'}^*e_{\lambda}=\phi_{\lambda}(t_{w'}^{\lambda},t_{w'}^{\lambda})e_{\lambda}=e_{\lambda}.
	\end{equation}
    For $t_w^{\lambda}$, we have
    \begin{equation*}
    \begin{aligned}
    	\phi_{\lambda}(t_{w}^{\lambda},t_{w}^{\lambda})e_{\lambda}&=e_{\lambda}\psi_{w}\psi_{w}^*e_{\lambda}\\
    	&=e_{\lambda}\psi_{w'}\psi_{u_k}\psi_{u_k}\psi_{w'}^*e_{\lambda}\\
    	&=\psi_{w'}\psi_{u_k}^2e(t_{w'}^{\lambda})\psi_{w'}^*.	
    \end{aligned}        
    \end{equation*}

     We next show that $\psi_{u_k}^2e(t_{w'}^{\lambda})=e(t_{w'}^{\lambda})$. Let $e(t_{w'}^{\lambda})=e(i_1,i_2,\dots,i_n)$. By definition, $(i_1,i_2,\dots,i_n)$ is a sequence consisting of $i_1,i_1-1,\dots,i_1-a+1,i_2,i_2-1,\dots, i_2-n+a-1$ if there are $a$ nodes in $\lambda^{(1)}$ and $n-a$ nodes in $\lambda^{(2)}$. As $|i_1-i_2|\geq n$, they are all distinct numbers. So the position of $l$ in the tableau $t_{w'}^{\lambda}$ is uniquely determined by $i_l$. So for any value of $u_k$, $i_{u_k}\neq i_{u_k+1}$. As $t_{w'}^{\lambda}$ is a standard tableau, $u_k$ is not in the node below $u_{k+1}$, so $i_{u_k}\neq i_{u_k+1}-1$. As $t_{w}^{\lambda}=t_{w'}^{\lambda}s_{u_k}$ is standard, $u_k$ is not in the node above $u_k+1$, so $i_{u_k}\neq i_{u_k+1}+1$, either. By $(\ref{25eq})$, $\psi_{u_k}^2e(t_{w'}^{\lambda})=e(t_{w'}^{\lambda})$. We have
     \begin{equation*}
         \begin{aligned}
     	\phi_{\lambda}(t_{w}^{\lambda},t_{w}^{\lambda})e_{\lambda}&=\psi_{w'}\psi_{u_k}^2e(t_{w'}^{\lambda})\psi_{w'}^*\\
     	&=\psi_{w'}e(t_{w'}^{\lambda})\psi_{w'}^*\\
     	&=\phi_{\lambda}(t_{w'}^{\lambda},t_{w'}^{\lambda})e_{\lambda}\\
     	&=e_{\lambda}.
     \end{aligned}     
     \end{equation*}

     So $\phi_{\lambda}(t_{w}^{\lambda},t_{w}^{\lambda})=1$. 
     
     For two different standard tableaux, $t,s\in Std(\lambda)$, let $k$ be the smallest number such that $t(k)\neq s(k)$. As $|j_k-j_l|\geq n$, $Res^{\Lambda}(t(k))\neq Res^{\Lambda}(s(k))$. Thus $e(t)\neq e(s)$. So we have
     \begin{equation*}
     	e_{\lambda}\psi_{d(t)}\psi_{d(s)}^*e_{\lambda}=\psi_{d(t)}e(t)e(s) \psi_{d(s)}^*=0
     \end{equation*}
     which implies that $\phi_{\lambda}(t,s)=0$.
     
     Therefore, $\phi_{\lambda}(t,s)=\delta_{st}$ for all $s,t\in Std(\lambda)$. Thus, $rad(\phi_{\lambda})=0$ for all $\lambda \in \mathfrak{B}_n^{(r)}-\mathfrak{D}_n^{(r)}$ and $TL_{r,1,n}^{\Lambda}$ is semi-simple.
     
     We now turn to the necessity. Without losing generality, let $j_1$ and $j_2$ be the two closest elements in all the $j$'s and $0<j_2-j_1=b<n$. If $e>0$, let $b$ be such that $2b\leq e$. 
     
     Let $\lambda=((1^{n-b}),(1^b),\emptyset,\dots,\emptyset)$ and $t_0$ be the standard tableau of shape $\lambda$ such that $1,2,\dots,b$ are in the second component. We claim that $C_{t_0}\in rad(\phi_{\lambda})$, which implies that $TL_{r,1,n}^{\Lambda}$ is not semi-simple according to Theorem $\ref{33}$.
     
     We first prove $\phi_{\lambda}(t_0,s)=0$ if $s\neq t_0$. By definition,
     \begin{equation*}
     	\phi_{\lambda}(t_0,s)e_{\lambda}=e_{\lambda}\psi_{d(t_0)}^*\psi_{d(s)}e_{\lambda}=\psi_{d(t_0)}^*e(t_0) e(s)\psi_{d(s)}.
     \end{equation*}
     As $e(\mathbf{i})$'s are orthogonal idempotents, $e(t_0) e(s)\neq 0$ if and only if $e(t_0)=e(s)$. Without losing generality, let $j_1=0$ and $j_2=b$. Then $e(t_0)=e(b,b-1,b-2,\dots, b-n)$. We prove that for $s=t_0$ if $e(t_0)=e(s)$ and $s\in Std(\lambda)$. It is enough to show that $1,2,\dots,b$ are in the second component of $s$. This is implied by the fact $Res_s^{\Lambda}(i)=b+1-i$ for $1\leq i\leq b$ and $s$ is standard. So if $s\neq t_0$, $\phi_{\lambda}(t_0,s)=0$.
     
     We next prove $\phi_{\lambda}(t_0,t_0)=0$. Let $d(t_0)=s_{u_1}s_{u_2}\dots s_{u_k}$ be the official reduced expression and $t_1$ be the standard tableau such that $d(t_1)=s_{u_1}s_{u_2}\dots s_{u_{k-1}}$. As all of $1,2\dots b$ are in the second component of $t_0$ and both $t_0$ and $t_1$ are standard, $s_{u_k}=s_b$ and $t_1$ is the standard tableau with $1,2,\dots,b-1,b+1$ in the second component. We have
     \begin{equation*}
     	e(t_1)=e(b,b-1,\dots,2,0,1,-1,\dots,b-n).
     \end{equation*}
     And we have
     \begin{equation*}
         \begin{aligned}
     	&e_{\lambda}\psi_{d(t_0)}^*\psi_{d(t_0)}e_{\lambda}\\
     	&=e_{\lambda}\psi_{d(t_1)}^*\psi_{u_k}^2 \psi_{d(t_1)}e_{\lambda}\\
     	&=\psi_{d(t_1)}^*\psi_{b}^2 e(b,b-1,\dots,2,0,1,-1,\dots,b-n)\psi_{d(t_1)}\\
     	&=\psi_{d(t_1)}^*(y_{b+1}-y_b) e(b,b-1,\dots,2,0,1,-1,\dots,b-n)\psi_{d(t_1)}\\
     	&=\psi_{d(t_1)}^*\psi_{d(t_1)}(y_c-y_d)e_{\lambda}
     \end{aligned}     
     \end{equation*}

        for some $c$ and $d$ between $b+1$ and $1-b$. The second last equation comes from the third case in $(\ref{25eq})$. The last equation is from $(\ref{21eq}),(\ref{23eq})$ and $(\ref{24eq})$ where the case $i_r=i_{r+1}$ is excluded by the fact that $2b\leq e$ and only $s_l$ where $l\leq 2b-1$ appears in the official reduced expression of $d(t_1)$.
        
        And by Lemma \ref{lrh lm}, we have
        \begin{equation*}
        	y_r e_{\lambda}=y_s e_{\lambda}=0\text{    mod } TL_{r,1,n}^{\triangleleft\lambda}.
        \end{equation*}
     Thus,
     \begin{equation*}
     	e_{\lambda}\psi_{d(t_0)}^*\psi_{d(t_0)}e_{\lambda}=0\text{    mod } TL_{r,1,n}^{\triangleleft\lambda}.
     \end{equation*}
     Therefore, we have $\phi_{\lambda}(t_0,t_0)=0$ by definition. In conclusion, $C_{t_0}\in rad(\phi_{\lambda})$, which implies that $TL_{r,1,n}^{\Lambda}$ is not semi-simple.
\end{proof}

The following corollary is an immediate consequence of Theorem $\ref{Th35}$ and $(\ref{32eq})$:

\begin{cor} \label{coss}
	Let $TL_{r,1,n}$ be the generalised Temperley-Lieb quotient defined in Definition \ref{TL} with parameters $q,v_1,v_2,\dots,v_r$. Then $TL_{r,1,n}(q,v_1,v_2,\dots,v_r)$ is semisimple if and only if  
	\begin{equation*}
		\dfrac{v_i}{v_j}\neq q^l
	\end{equation*}
	for all $1\leq i\leq j\leq r$ and $-n+1\leq l\leq n-1$ and
	\begin{equation*}
	    1+q+q^2+\dots+q^i\neq 0
	\end{equation*}
	for $1\leq i\leq n-1$.
\end{cor}
Comparing with Theorem \ref{ssth}, the generalised Temperley-Lieb algebra $TL_{r,1,n}$ is semisimple if and only if the
same is true of the corresponding Hecke algebra $H(r,1,n)$.

\section{Irreducible representations and decomposition numbers}
In this section, we study the irreducible representations and the decomposition numbers for cell modules of $TL_{r,1,n}$. Let $\lambda$ be a one-column multipartition of $n$ consisting of at most two non-empty components. Let $W(\lambda)$ be the cell module of $TL_{r,1,n}$ corresponding to $\lambda$ and $L(\lambda)=W(\lambda)/rad(\phi_{\lambda})$ be the corresponding simple module. We first calculate the dimension of $L(\lambda)$. By Theorem $\ref{33}$, it is enough to find the rank of the bilinear form $\phi_{\lambda}$.


\subsection{Garnir tableaux}
When calculating the value of the bilinear form $\phi_{\lambda}(s,t)$ for $s,t\in Std(\lambda)$, we can meet some non-standard tableaux. Inspired by Lobos and Ryom-Hansen, we use Garnir tableaux as a tool to deal with these non-standard tableaux. This method is originally due to Murphy in \cite{murphy1995representations}. We start with the definition of a Garnir tableau.
 
 \begin{defn}
 	Let $\lambda \in \mathfrak{B}_n^{(r)}-\mathfrak{D}_n^{(r)}$ be a one-column multipartition of $n$ and $g$ be a $\lambda$-tableau. We call $g$ a Garnir tableau if there exists $k$ with $1\leq k\leq n-1$ such that
 	
 	(a) $g$ is not standard but $g\circ s_k$ is;
 	
 	(b) if $g\circ s_i \triangleleft g$, then $i=k$, where $\triangleleft$ is the partial order defined by (\ref{potab}).
 \end{defn}

 Equivalently, $g$ is a Garnir tableau if and only if there is a unique $k,\;\;1\leq k\leq n-1$ such that $g(k)>g(k+1)$ with respect to the order on nodes and this number $k$ is in the same column as $k+1$. Here are some Garnir tableaux of shape $\lambda=((1^2),(1^4),\emptyset)$:
 
 \begin{figure}[h]
 	\begin{center}
 		\begin{tikzpicture}[scale=0.6]
 			\foreach \x in {1,2,8,9,15,16}
 			\draw (\x,2)--(\x,4);
 			\foreach \x in {3,4,10,11,17,18}
 			\draw (\x,0)--(\x,4);
 			\foreach \x in {2,3,4}
 			\draw (1,\x)--(2,\x);
 			\foreach \x in {0,1,2,3,4}
 			\draw (3,\x)--(4,\x);
 			\foreach \x in {2,3,4}
\draw (8,\x)--(9,\x);
\foreach \x in {0,1,2,3,4}
\draw (10,\x)--(11,\x);
 			\foreach \x in {2,3,4}
\draw (15,\x)--(16,\x);
\foreach \x in {0,1,2,3,4}
\draw (17,\x)--(18,\x);
 			
 	    	\draw node at (1.5,3.5){2};\draw node at (1.5,2.5){1};\draw node at (3.5,3.5){3};\draw node at (3.5,2.5){4};\draw node at (3.5,1.5){5};\draw node at (3.5,0.5){6};
\draw node at (8.5,3.5){3};\draw node at (8.5,2.5){2};\draw node at (10.5,3.5){1};\draw node at (10.5,2.5){4};\draw node at (10.5,1.5){5};\draw node at (10.5,0.5){6};
\draw node at (15.5,3.5){1};\draw node at (15.5,2.5){4};\draw node at (17.5,3.5){3};\draw node at (17.5,2.5){2};\draw node at (17.5,1.5){5};\draw node at (17.5,0.5){6};
 			\draw node at (5.5,3.5){$\emptyset$};	\draw node at (12.5,3.5){$\emptyset$};	\draw node at (19.5,3.5){$\emptyset$};
 			\draw node at (6.5,2.5){;};	\draw node at (13.5,2.5){;};	\draw node at (20.5,2.5){.};
 			
 		\end{tikzpicture}
 	
 	\end{center}
 	\caption{Garnir tableaux of $\lambda$ } 
 
 \end{figure}
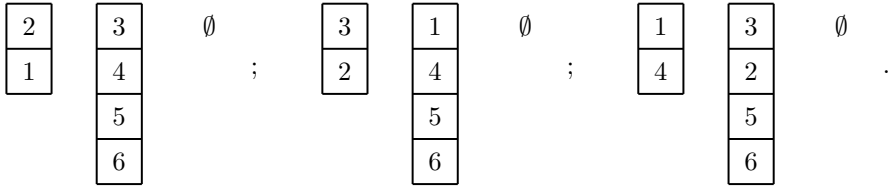
 	Comparing the three tableaux, we find that a Garnir tableau is neither uniquely determined by the positions nor by the numbers that cause the ``non-standardness". The next lemma shows that any non-standard tableaux can be transformed into a Garnir one.
 
 \begin{lem} \cite[Corollary 22]{LOBOSMATURANA2020106277}$\label{41}$
 	Let $t$ be a non-standard tableau of shape $\lambda$. Then there exists a Garnir tableau $g$ and $w\in \mathfrak{S}_n$ such that $t=g\circ w$ and $l(d(t))=l(d(g))+l(w)$.
 \end{lem}
Fix a Garnir tableau $g$ and a dominant weight $\Lambda$, let $e(g):= e(res^{\Lambda}(g))$ be the idempotent of $TL_{r,1,n}^{\Lambda}$ where $res^{\Lambda}(g)$ is as defined in $(\ref{eqres})$. The next lemma shows that $e(g)$ is a composition of lower terms.
 \begin{lem} \cite[Lemma 35]{LOBOSMATURANA2020106277}$\label{42}$
 	If $g$ is a Garnir tableau of shape $\lambda$, let $e(g)=e(Res^{\Lambda}(g))$ then
 	\begin{equation}
 		e(g)=\sum_{\mu \triangleleft \lambda} D_{\mu}
 	\end{equation}
 where $D_{\mu}$ is in the two-sided ideal of $H(r,1,n)$ generated by $e_{\mu}$, which is the idempotent defined before Definition $\ref{Df28}$.
 \end{lem}
 These two lemmas provide us a method to transform $e(t)$ into a combination of lower terms when $t$ is not a standard tableau.
 
\subsection{The bilinear form $\phi$ and the irreducible representations of $TL_{r,1,n}$} \label{6.2}
 
This subsection concerns the dimensions of the simple modules $L(\lambda)$ of $TL_{r,1,n}$. According to Theorem \ref{33}, $dim(L(\lambda))$ equals to the rank of the bilinear form $\phi_{\lambda}(,)$.
 We next introduce a theoretical method to calculate the value of the bilinear form $\phi_{\lambda}(s,t)$ for $s,t\in Std(\lambda)$. In fact, we will not calculate any values of  $\phi_{\lambda}(s,t)$. This method is used to show the value of the bilinear form $\phi_{\lambda}(s,t)$ equals to the one on the corresponding cell module of $TLB(q,Q)$.
 
 We first introduce a notation which will be used when calculating the values of $\phi_{\lambda}(s,t)$. Denote $N=\{1,2,\dots,n-1\}$. For $U^{(k)}=(U^{(k)}_1,U^{(k)}_2,\dots,U^{(k)}_k)\in N^k$ where $k$ is a non-negative integer, let $\psi_{U^{(k)}}=\psi_{U^{(k)}_1}\psi_{U^{(k)}_2}\dots \psi_{U^{(k)}_k}$ with $\psi_i(i\in N)$ being the KLR generators in Definition $\ref{klrdf}$. We use the following equivalence relation to describe the difference between two elements in the cyclotomic KLR algebra.

  \begin{defn} \label{dfequiv}
  	For any positive integer $k$, define an equivalence relation $\stackrel{k}{\sim}$ on $N^k$ as follows: 
  We write $V^{(k)}\stackrel{k}{\sim}U^{(k)}$ if 
  	\begin{equation*}
  		\psi_{V^{(k)}}^*e_{\lambda}-\psi_{U^{(k)}}^*e_{\lambda}=\sum_{l<k} c(W^{(l)})\psi_{W^{(l)}}^*e_{\lambda}
  	\end{equation*}
     where $c(W^{(l)})\in R$ and $W^{(l)}$ runs over $N^l$ for all $0\leq l<k$.
  \end{defn}
  The following lemma shows that two reduced expressions of the same element lead to equivalent sequences.
\begin{lem} $\label{45}$
	For $w\in \mathfrak{S}_n$, let $s_{V^{(k)}}$ and $s_{U^{(k)}}$ be two reduced expressions of $w$ where $k=l(w)$, then $V^{(k)}\stackrel{k}{\sim}U^{(k)}$.
\end{lem}
\begin{proof}
	As both $s_{V^{(k)}}$ and $s_{U^{(k)}}$ are reduced expressions of the same element $w$, they can be transformed to each other by braid relations. These relations correspond to $(\ref{22eq})$ and $(\ref{26eq})$ in $TL_{r,1,n}^{\Lambda}$. The error terms only occur in $(\ref{26eq})$. They lead to strictly shorter sequences. By Definition \ref{dfequiv}, we have $V^{(k)}\stackrel{k}{\sim}U^{(k)}$.
\end{proof}
The following lemma provides us the main tool to eliminate the lower terms when calculating the value of the bilinear form.
\begin{lem} $\label{algo}$
	For any non-negative integer $k$, let $U^{(k)}$ be a sequence of length $k$ and $\lambda\in \mathfrak{B}_n^{(r)}-\mathfrak{D}_n^{(r)}$. Then $\psi_{U^{(k)}}^* e_{\lambda}$ can be transformed into the following form using the generating relations $(\ref{18eq})-(\ref{26eq})$:
	\begin{equation} \label{89}
		\psi_{U^{(k)}}^*e_{\lambda}=\sum_{s\in Std(\lambda)}c_{U^{(k)}}(s)\psi_{d(s)}^* e_{\lambda}+\sum_{i=1}^{n} Y_i+\sum_{g\in Gar(\lambda)} D_g
	\end{equation}
	where $Y_i$ is in the two-sided ideal generated by $y_ie_{\lambda}$ and $D_g$ is in the one generated by $e(g)$ with $g$ running over the Garnir tableaux of shape $\lambda$.
	
\end{lem}
\begin{proof}
	We prove this by induction on $k$. If $k=0$, we can get the form directly. Assume this is true for any sequence $U^{(l)}$ with $l<k$ for a positive integer $k$. For a sequence $U^{(k)}$, we consider three cases:
	
	1. $s_{U^{(k)}}$ is a reduced expression of some $w\in \mathfrak{S}_n$ and $t^{\lambda}\circ w$ is a standard tableau. Denote by $s$ the standard tableau $t^{\lambda}\circ w$. Let $V^{(k)}$ be the sequence such that $	\psi_{V^{(k)}}^*$ is the chosen $\psi_{d(s)}^*$ in the definition of the cellular basis. Then $s_{V^{(k)}}$ and $s_{U^{(k)}}$ are two reduced expressions of $w$ with $k=l(w)$, so we have $V^{(k)}\stackrel{k}{\sim}U^{(k)}$ by Lemma $\ref{45}$. Thus,
	\begin{equation}
		\psi_{U^{(k)}}^*e_{\lambda}=\psi_{d(s)}^*e_{\lambda}+\sum_{l<k} c(W^{(l)})\psi_{W^{(l)}}^*e_{\lambda}.
	\end{equation}
    By the inductive assumption, $\psi_{W^{(l)}}^*e_{\lambda}$ is of the form in $(\ref{89})$.
    Therefore, $\psi_{U^{(k)}}^*e_{\lambda}$ can be transformed into the form in $(\ref{89})$.
    
    2. $s_{U^{(k)}}$ is a reduced expression of some $w\in \mathfrak{S}_n$, but $t^{\lambda}\circ w$ is not a standard tableau. By Lemma $\ref{41}$, there exists a Garnir tableau $g$ and an element $w_1\in \mathfrak{S}_n$ such that $t^{\lambda}\circ w=g\circ w_1$. Let $V(g)$ and $V(w_1)$ be two sequences consisting of numbers in $\{1,2,\dots,n-1 \} $ such that $s_{V(g)}$ and $s_V(w_1)$ are reduced expressions of $d(g)$ and $w_1$ respectively. Let $V^{(k)}=V(g)V(w_1)$ be the combination of these two sequences. Then $s_{V^{(k)}}$ is a reduced expression of $w$ as well. By Lemma $\ref{45}$, $V^{(k)}\stackrel{k}{\sim}U^{(k)}$. So we have
    \begin{equation*}
    \begin{aligned}
    	\psi_{U^{(k)}}^*e_{\lambda}
    	&=\psi_{V^{(k)}}^*e_{\lambda}+\sum_{l<k} c(W^{(l)})\psi_{W^{(l)}}^*e_{\lambda}\\
    	&=\psi_{V(w_1)}^*\psi_{V(g)}^*e_{\lambda}+\sum_{l<k} c(W^{(l)})\psi_{W^{(l)}}^*e_{\lambda}\\
    	&=\psi_{V(w_1)}^*e(i^g)\psi_{V(g)}^*+\sum_{l<k} c(W^{(l)})\psi_{W^{(l)}}^*e_{\lambda}	.
    \end{aligned}        
    \end{equation*}

    The first term is in the two-sided ideal generated by $e(i^g)$ and the other terms satisfy that the length of the sequence decreases strictly, so the inductive assumption can be used.
    
    3. If $s_{U^{(k)}}$ is not a reduced expression of any $w\in \mathfrak{S}_n$, let $l$ be the largest number such that $s_{U^{(k)}|_l}$ is a reduced expression of some $w\in \mathfrak{S}_n$ where $U^{(k)}|_l$ is the sub-sequence of $U^{(k)}$ consisting of the first $l$ terms. By the exchange condition, there exists a sequence $V^{(l)}$ ending with $U^{(k)}_{l+1}$ such that $s_{V^{(l)}}$ is a reduced expression for $s_{U^{(k)}|_l}$. By Lemma $\ref{45}$, $V^{(l)}\stackrel{k}{\sim}U^{(k)}|_l$. Therefore, we have
    \begin{equation*}
       \begin{aligned}
    	\psi_{U^{(k)}}^*e_{\lambda}&=(\psi_{U^{(k)}|_l}\psi_{U^{(k)}_{l+1}}\psi_{U^{(k-l-1)}})^*e_{\lambda}\\
    	&=(\psi_{V^{(l)}}\psi_{U^{(k)}_{l+1}}\psi_{U^{(k-l-1)}})^*e_{\lambda}+\sum_{l<k} c(W^{(l)})\psi_{W^{(l)}}^*e_{\lambda}\\
    	&=(\psi_{V^{(l)}|_{l-1}}\psi_{U^{(k)}_{l+1}}^2\psi_{U^{(k-l-1)}})^*e_{\lambda}+\sum_{l<k} c(W^{(l)})\psi_{W^{(l)}}^*e_{\lambda}\\
    	&=\psi_{U^{(k-l-1)}}^*e(i)\psi_{U^{(k)}_{l+1}}^2\psi_{V^{(l)}|_{l-1}}^*+\sum_{l<k} c(W^{(l)})\psi_{W^{(l)}}^*e_{\lambda}	.
    \end{aligned}     
    \end{equation*}

    For the same reason as above, we only need to deal with the first term. By ($\ref{25eq}$), $e(i)\psi_{U^{(k)}_{l+1}}^2$ is $0$, $e(i)$, $	(y_{s+1}-y_s)e(i) $ or 	$(y_s-y_{s+1})e(i) $. The first case is trivial. The second case leads to a strictly shorter sequence. If we get $	(y_{s+1}-y_s)e(i) $ or 	$(y_s-y_{s+1})e(i) $, ($\ref{23eq}$) can help to move $y_s$ to the right to get some elements $Y_i$. The error term will also lead to a strictly shorter sequence which is covered by the inductive assumption.
    
    Therefore, for any sequence $U^{(k)}$ of length $k$, $\psi_{U^{(k)}}$ can be transformed into the form in $(\ref{89})$. Thus the lemma has been proved.
\end{proof}

It should be remarked that the form we get in the lemma above does not depend on the dominant weight $\Lambda$, given the fact that only the generating relations $(\ref{18eq})-(\ref{26eq})$ are used in the procedure. Therefore, this lemma builds a bridge between a general $TL_{r,1,n}^{\Lambda}$ and a Temperley-Lieb algebra of type $B_n$.

To be more precise, let $\Lambda=\Lambda_{i_1}+\Lambda_{i_2}+\dots +\Lambda_{i_r}$ be a dominant weight and $TL_{r,1,n}^{\Lambda}$ be the corresponding generalised Temperley-Lieb algebra. Let $\lambda \in \mathfrak{B}_n^{(r)}-\mathfrak{D}_n^{(r)}$ be a multipartition in which the $u^{th}$ and $v^{th}$ components are non-empty. Denote by $W(\lambda)$ the cell module of $TL_{r,1,n}^{\Lambda}$ corresponding to $\lambda$ and by $\phi_{\lambda}$ the bilinear form on $W(\lambda)$.

Let $\Lambda'=\Lambda_{i_u}+\Lambda_{i_v}$ and $TLB_n^{\Lambda'}$ be the corresponding Temperley-Lieb algebra of Type $B_n$. Let $\lambda'$ be the bipartition with the same shape as the non-empty part of $\lambda$. For example, if $\lambda=((1^a),\emptyset,(1^{n-a}),\emptyset,\dots,\emptyset)$ is the multipartition of $n$ with $a$ nodes in the first partition and $n-a$ in the third, the corresponding bipartition $\lambda'=((1^a),(1^{n-a}))$ and the dominant weight $\Lambda'=\Lambda_{i_1}+\Lambda_{i_3}$.

Denote by $W(\lambda')$ and $\phi_{\lambda'}$ the corresponding cell module of $TLB_n^{\Lambda'}$ and the bilinear form. Let $f$ be the natural map from $Std(\lambda)$ to $Std(\lambda')$. Then we have

\begin{cor} $\label{co47}$
	$\phi_{\lambda}(s,t)=\phi_{\lambda'}(f(s),f(t))$, for all $s,t\in Std(\lambda)$.
\end{cor}
\begin{proof}
	By definition, we have
	\begin{equation*}
		\phi_{\lambda}(s,t) e_{\lambda}= e_{\lambda}\psi_{d(s)}\psi_{d(t)}^*e_{\lambda}\text{    mod } TL_{r,1,n}^{\triangleleft\lambda}.
	\end{equation*}
     and
	\begin{equation*}
	\phi_{\lambda'}(f(s),f(t)) e_{\lambda'}= e_{\lambda'}\psi_{d(f(s))}\psi_{d(f(t))}^*e_{\lambda'}\text{    mod } TLB_n^{\triangleleft\lambda'}.
    \end{equation*}     
    
     The right hand sides of the two equations share the same expression. But this is not enough to show $\phi_{\lambda}(s,t)=\phi_{\lambda'}(f(s),f(t))$ because $TL_{r,1,n}^{\triangleleft\lambda}\neq TLB_n^{\triangleleft\lambda'}$. By the fact that $e(i)$'s are orthogonal idempotents, we have $ e_{\lambda}\psi_{d(s)}\psi_{d(t)}^*e_{\lambda}=\psi_{d(s)}\psi_{d(t)}^*e_{\lambda} $ in both $TL_{r,1,n}^{\Lambda}$ and $TLB_n^{\Lambda'}$ if $ e_{\lambda}\psi_{d(s)}\psi_{d(t)}^*e_{\lambda}\neq 0 $ .
     Lemma $(\ref{algo})$ implies that the right hand sides of the two equations can be transformed into the same form:
     	\begin{equation} 
     	\psi_{d(s)}\psi_{d(t)}^*e_{\lambda}=c_{s,t} e_{\lambda}+\sum_{i=1}^{n} Y_i+\sum_{g\in Gar(\lambda)} D_g.
     \end{equation}
     
     As the transformation only depends on the generating relations $(\ref{18eq})-(\ref{26eq})$ which are shared by $TL_{r,1,n}^{\Lambda}$ and $TLB_n^{\Lambda'}$, we have $c_{s,t}=c_{f(s),f(t)}$.
     Lemma $\ref{lrh lm}$ and Lemma $\ref{42}$ show that $Y_i$ and $D_g$ are in $TL_{r,1,n}^{\triangleleft\lambda}$ and $TLB_n^{\triangleleft\lambda'}$. So we have
     \begin{equation*}
     	\phi_{\lambda}(s,t)=c_{s,t}
     \end{equation*}
      and
      \begin{equation*}
      	\phi_{\lambda'}(f(s),f(t))=c_{f(s),f(t)}.
      \end{equation*}
      Therefore, we have
           \begin{equation*}
      	\phi_{\lambda}(s,t)=\phi_{\lambda'}(f(s),f(t)).
      \end{equation*}
\end{proof}

Further, let $L(\lambda)$ and $L(\lambda')$ be the simple modules of $TL_{r,1,n}^{\Lambda}$ and $TLB_n^{\Lambda'}$ respectively in Theorem $\ref{33}$. As a direct consequence, we have
\begin{cor}
	$dim(L(\lambda))=dim(L(\lambda'))$.
\end{cor}

\subsection{Decomposition numbers}
We next concentrate on the decomposition numbers of $TL_{r,1,n}^{\Lambda}$. In a similar way to the observation above, we claim that the decomposition numbers can be obtained from those of $TLB_n^{\Lambda'}$ where $\Lambda'$ is the dominant weight defined in terms of the cell module $C_{\lambda}$ in the last subsection. To prove this, we show that the cell module $W(\lambda)$ and the simple module $L(\lambda)$ of $TL_{r,1,n}^{\Lambda}$ are isomorphic to those of a Temperley-Lieb algebra of type $B_n$ as $TLB_n^{\Lambda'}$-modules.


\begin{lem}
\label{tlquot}
	$TLB_n^{\Lambda'}$ is a quotient of $TL_{r,1,n}^{\Lambda}$ by the two-sided ideal generated by all $e(\mathbf{i})=e(\mathbf{i}_1,\mathbf{i}_2,\dots,\mathbf{i}_n)$ such that $(\Lambda',\alpha_{\mathbf{i}_1})=0$.
\end{lem}

\begin{proof}
	By comparing the generators and relations in Definition $\ref{klrdf}$, we see that $\mathcal{R}_n^{\Lambda'}$ is a quotient of $\mathcal{R}_n^{\Lambda}$ by the two-sided ideal generated by all $e(\mathbf{i})$ such that $(\Lambda',\alpha_{\mathbf{i}_1})=0$. 
	
	Fix an idempotent $e(i_1,i_2,i_3,\mathbf{j}')\in \mathcal{R}_n^{\Lambda}$ where $(\alpha_{i_a},\Lambda)>0$ for $a=1,2,3$. We next show  $ e(i_1,i_2,i_3,\mathbf{j}')$ is in the two-sided ideal of $\mathcal{R}_n^{\Lambda}$, 
	\[
	\langle e(\mathbf{i}),  e(i,i+1,\mathbf{j}) , e(i_1,i_2,i_3,\mathbf{j}')|(\alpha_i,\Lambda')>0,  
	(\alpha_{i_j},\Lambda')>0 \rangle,
	\]
	where  $\mathbf{i}$ runs over $I^n$ such that $(\Lambda',\alpha_{\mathbf{i}_1})=0$. 

	By Theorem $\ref{dftl}$ we have
	\begin{equation*}
			\begin{aligned}
			TLB_n^{\Lambda'}&=\mathcal{R}_n^{\Lambda'}/\langle  e(i,i+1,\mathbf{j}) , e(i_1,i_2,i_3,\mathbf{j}') |(\alpha_i,\Lambda')>0, (\alpha_{i_j},\Lambda')>0\text{ for }j=1,2,3\rangle\\
			    &=\mathcal{R}_n^{\Lambda}/\langle e(\mathbf{i}),  e(i,i+1,\mathbf{j}) , e(i_1,i_2,i_3,\mathbf{j}')|(\alpha_i,\Lambda')>0, (\alpha_{i_j},\Lambda')>0 \rangle
		\end{aligned}    
	\end{equation*}

	where $\mathbf{i}$ runs over $I^n$ such that $(\Lambda',\alpha_{\mathbf{i}_1})=0$.
	Let $f$ be the quotient map from $\mathcal{R}_n^{\Lambda}$ to $TLB_n^{\Lambda'}$. It is enough to show that $f( e(i_1,i_2,i_3,\mathbf{j}'))=0$. By the same method as that in the proof of Lemma \ref{triquo}, we can get $f( e(i_1,i_2,i_3,\mathbf{j}'))C_{s,t}^{\lambda}=0$ where $C_{s,t}^{\lambda}$ runs over the cellular basis of $TLB_n^{\Lambda'}$ in Theorem \ref{cb33}. So $ e(i_1,i_2,i_3,\mathbf{j}')$ is in the corresponding two-sided ideal of $\mathcal{R}_n^{\Lambda}$, $\langle e(\mathbf{i}),  e(i,i+1,\mathbf{j}) , e(i_1,i_2,i_3,\mathbf{j}')|(\alpha_i,\Lambda')>0, (\alpha_{i_j},\Lambda')>0 \rangle$. 
	
	Therefore, we have
	\begin{equation*}
	\begin{aligned}
	    TLB_n^{\Lambda'}&=\mathcal{R}_n^{\Lambda}/\langle e(\mathbf{i}),  e(i,i+1,\mathbf{j}) , e(i_1,i_2,i_3,\mathbf{j}')|(\alpha_i,\Lambda')>0, (\alpha_{i_j},\Lambda')>0 \rangle\\
	    &=\mathcal{R}_n^{\Lambda}/\langle e(\mathbf{i}),  e(i,i+1,\mathbf{j}) , e(i_1,i_2,i_3,\mathbf{j}')|(\alpha_i,\Lambda)>0, (\alpha_{i_j},\Lambda)>0 \rangle\\
			    &=TL_{r,1,n}^{\Lambda}/\langle e(\mathbf{i}) \rangle
		\end{aligned}	    
	\end{equation*}

	where $\mathbf{i}$ runs over $I^n$ such that $(\Lambda',\alpha_{\mathbf{i}_1})=0$.
\end{proof}

We next show that further, the cell module $W(\lambda)$ can be regarded as a module of the 
quotient algebra $TLB_n^{\Lambda'}$.
\begin{lem} $\label{lm50}$
	Let $W(\lambda)$ be the cell module of $TL_{r,1,n}^{\Lambda}$ defined above. Then $W(\lambda)$ is a $TLB_n^{\Lambda'}$-module.
\end{lem}
\begin{proof}
	For $t\in Std(\lambda)$, let $e(t)=e(j_1,j_2,\dots,j_n)$. Then by definition, $j_1=i_u\text{ or }i_v$. For any $e(\mathbf{i})$ such that $(\Lambda',\alpha_{\mathbf{i}_1})=0$, we have $\mathbf{i}_1\neq i_u\text{ or }i_v$. So
	\begin{equation}
		e(\mathbf{i})\psi_{d(t)}^*e_{\lambda}=e(\mathbf{i})e(t)\psi_{d(t)}^*=0
	\end{equation}
    which implies $r_{	e(\mathbf{i})}(t',t)=0$ for all $t,t'\in Std (\lambda)$. Therefore,
    \begin{equation}
    	e(\mathbf{i})W(\lambda)=0
    \end{equation}
    for all $e(\mathbf{i})$ such that $(\Lambda',\alpha_{\mathbf{i}_1})=0$. 
    That is, the ideal which defines $TLB_n^{\Lambda'}$ as a quotient of $TL_{r,1,n}^{\Lambda}$
    acts trivially on $W(\lambda)$. Hence $W(\lambda)$ is a $TLB_n^{\Lambda'}$-module.
\end{proof}

As a direct consequence, we have that the simple module $L(\lambda)=W(\lambda)/rad(\phi_{\lambda})$ is a module of $TLB_n^{\Lambda'}$:
\begin{lem}\label{lm51}
	Let $L(\lambda)$ be the simple module of $TL_{r,1,n}^{\Lambda}$. Then $L(\lambda)$ is a $TLB_n^{\Lambda'}$-module.
\end{lem}
 Further, the next lemma shows that the two $TLB_n^{\Lambda'}$-modules obtained in the last two lemmas are exactly the corresponding cell and simple modules corresponding to an appropriate $\lambda'$.
\begin{lem} $\label{lm52}$
	Let $\lambda'$ be the bipartition with the same shape as the non-empty part of $\lambda$ and $W(\lambda')$, $L(\lambda')$ be the cell module and simple module of $TLB_n^{\Lambda'}$ corresponding to $\lambda'$ respectively. Then $W(\lambda')\simeq W(\lambda)$, $L(\lambda')\simeq L(\lambda)$ as $TLB_n^{\Lambda'}$-modules.
\end{lem}
\begin{proof} 
	We first prove $W(\lambda')\simeq W(\lambda)$. Let $f: Std(\lambda)\to Std(\lambda')$ be the natural bijection and $F:W(\lambda)\to W(\lambda')$ be the induced map. We will prove that $F$ is an isomorphism between the two $TLB_n^{\Lambda'}$-modules. Since $F$ is an isomorphism of vector spaces, it is enough to show that $F$ is a module homomorphism, which is to show $F(aC_t)=aF(C_t)$ for all $a$ in the generating set of $TLB_n^{\Lambda'}$ and $t\in Std(\lambda)$. Let $r_a(s,t)$ and $r'_a(f(s),f(t))$ be the coefficients in the equation (\ref{55}) corresponding to $TL_{r,1,n}^{\Lambda}$ and $TLB_n^{\Lambda'}$ respectively. We next show that $r_a(s,t)=r'_a(f(s),f(t))$ for all $a$ in the generating set of $TLB_n^{\Lambda'}$ and deduce that $F$ is a module homomorphism.
	
	If $a=e(\mathbf{i})$, we have
	\begin{equation*}
		r_a(s,t)=r'_a(f(s),f(t))=
		\begin{cases}
			1 &\text{ if } s=t \text{ and } a=e(t)\\
			0 & \text{ otherwise.}
		\end{cases}
	\end{equation*}
	
	If $a=\psi_i$, for $t\in Std(\lambda)$, denote by $U^{(k)}$ be the sequence such that $\psi_{U_{i,t}^{(k)}}^*= \psi_i \psi_{d(t)}^*$. Then by Lemma $\ref{algo}$, $\psi_i \psi_{d(t)}^*e_{\lambda}$ and $\psi_i \psi_{d(f(t))}^*e_{\lambda'}$ can be written in the form of $(\ref{89})$ with
	\begin{equation*}
		c_{i,t}(s)=c_{i,f(t)}(f(s)).
	\end{equation*}
	So we have
	\begin{equation*}
		r_a(s,t)=c_{i,t}(s)=c_{i,f(t)}(f(s))=r'_a(f(s),f(t)).
	\end{equation*}
	
	If $a=y_i$, for $t\in Std(\lambda)$, by (\ref{21eq}),(\ref{23eq}) and (\ref{24eq}), we have
	\begin{equation}\label{ydecom}
		y_i\psi_{d(t)}^*e_{\lambda}= \psi_{d(t)}^*y_je_{\lambda}+\sum_{l(V)<l({d(t)})}c(V)\psi_{V}^*e_{\lambda},
	\end{equation}
	where $c(V)$ is 1 or -1.
	Then by Lemma $\ref{algo}$, $y_i \psi_{d(t)}^*e_{\lambda}$ and $y_i \psi_{d(f(t))}^*e_{\lambda'}$ can be written in the form of $(\ref{89})$ with the same coefficients. So we have 
	\begin{equation*}
		r_a(s,t)=c_{y_i,t}(s)=c_{y_i,f(t)}(f(s))=r'_a(f(s),f(t)).
	\end{equation*}
	
	To summarise, $r_a(s,t)=r'_a(f(s),f(t))$ for all $a$ in the generating set of $TLB_n^{\Lambda'}$. Thus we have $F(aC_t)=aF(C_t)$ which implies $F$ is an isomorphism between the two $TLB_n^{\Lambda'}$-modules. 
	
	By Corollary $\ref{co47}$, $\phi_{\lambda}(s,t)=\phi_{\lambda'}(f(s),f(t))$. Thus $f$ induces an isomorphism between the two $TLB_n^{\Lambda'}$-modules $L(\lambda')$ and $L(\lambda)$ as well. 
\end{proof}

The following theorem shows the connection between the decomposition numbers for a generalised Temperley-Lieb algebra and those for a blob algebra.

\begin{thm} $\label{thm53}$
	Let $TL_{r,1,n}^{\Lambda}$ be a generalised Temperley-Lieb algebra (cf. Theorem \ref{dftl}) and  $\lambda,\mu \in \mathfrak{B}_n^{(r)}-\mathfrak{D}_n^{(r)}$ be two multipartitions of which all the components except for the $u^{th}$ and $v^{th}$  are empty. Let $TLB_n^{\Lambda'}$ and $\lambda',\mu'$ be as defined before Corollary \ref{co47}. We have $[W(\lambda),L(\mu)]=[W(\lambda'),L(\mu')]_{TLB}$ where the left hand side is the decomposition number for $TL_{r,1,n}^{\Lambda}$ and the right hand side is that for $TLB_n^{\Lambda'}$.
\end{thm}

\begin{proof}
	Lemma $\ref{lm52}$ shows that $L(\mu)$ is a simple $TLB_n^{\Lambda'}$-module and $[W(\lambda),L(\mu)]_{TLB}=[W(\lambda'),L(\mu')]_{TLB}$ as $TLB_n^{\Lambda'}$-modules. The proof of Lemma $\ref{lm50}$ shows that the kernel of the quotient map from $TL_{r,1,n}^{\Lambda}$ to $TLB_n^{\Lambda'}$ acts trivially on both $ W(\lambda') $ and $ L(\mu') $. So $Hom_{TL_{r,1,n}^{\Lambda}} (W(\lambda),L(\mu))=Hom_{TLB_n^{\Lambda'}} (W(\lambda),L(\mu))$. Therefore,  we have $[W(\lambda),L(\mu)]=[W(\lambda'),L(\mu')]_{TLB}$.
\end{proof}

The non-graded decomposition numbers of the Temperley-Lieb algebras of type $B_n$ are given by Martin and Woodcock in \cite{martin_woodcock_2003} and Graham and Lehrer in \cite{GRAHAM2003479} and the graded ones are given by Plaza in \cite{PLAZA2013182}. We give an interpretation of their result.

\begin{thm} ($cf.$ \cite{GRAHAM2003479},Theorem 10.16)
	Let $TLB_n^{\Lambda'}$ the Temperley-Lieb algebra of type $B_n$ corresponding to the dominant weight $\Lambda'$ (cf. Theorem \ref{dftl}) and $\lambda',\mu'$ be two bipartitions of $n$. Then the decomposition numbers for the cell module $W(\lambda')$ of $TLB_n^{\Lambda'}$,
	\begin{equation}
		[W(\lambda'),L(\mu')]_{TLB}=
		\begin{cases}
			1 &\text{ if } \lambda' \unlhd\mu' \text{ and there exists } t\in Std(\lambda') \text{ such that } e(t)=e_{\mu'}\\
			0 &\text{ otherwise,}
		\end{cases}
	\end{equation}
	where $e(t)$ and $e_{\mu'}$ are the KLR generators in Definition \ref{dfei}.
\end{thm}

We can now determine the decomposition numbers for the generalised Temperley-Lieb algebra $TL_{r,1,n}$.

\begin{thm} \label{decomtlr1n}
	Let $TL_{r,1,n}^{\Lambda}$ be the generalised Temperley-Lieb algebra (cf. Definition \ref{TL} and Theorem \ref{dftl}) over a field $R$ of characteristic $0$, defined by a dominant weight $\Lambda$ which satisfies $(\ref{reslmd})$. According to Theorem $\ref{cb33}$, it is a graded cellular algebra with the cell datum $(\mathfrak{B}_n^{(r)}-\mathfrak{D}_n^{(r)},*,Std,C,deg)$.  For $\lambda,\mu \in \mathfrak{B}_n^{(r)}-\mathfrak{D}_n^{(r)}$, let $W(\lambda)$ be the cell module and $L(\mu)$ be the simple of  $TL_{r,1,n}^{\Lambda}$ corresponding to $\lambda$ and $\mu$, respectively. Then the decomposition number
		\begin{equation}
		[W(\lambda),L(\mu)]=
		\begin{cases}
			1 &\text{ if } \lambda \unlhd\mu \text{ and there exists } t\in Std(\lambda) \text{ such that } e(t)=e_{\mu}\\
			0 &\text{ otherwise,}
		\end{cases}
	\end{equation}
		where $e(t)$ and $e_{\mu'}$ are the KLR generators in Definition \ref{dfei}.
\end{thm}
\begin{proof}
	If the set of indices of non-empty components of $\mu$ is a subset of that of $\lambda$, Theorem $\ref{thm53}$ implies that $[W(\lambda),L(\mu)]=1$ if $ \lambda \unlhd\mu \text{ and there exists } t\in Std(\lambda) $ such that $ e(t)=e_{\mu}$. By comparing the dimensions, we can see that these 1's are the only non-zero decomposition numbers for $TL_{r,1,n}^{\Lambda}$. If the indices of non-empty components of $\mu$ are not covered by those of $\lambda$, by comparing the first two indices in $e(t)$, we notice there is no standard tableau $t$ of shape $\lambda$ such that $e(t)=e_{\mu}$.
\end{proof}

It should be remarked that the former theorem shows that for any $\lambda,\mu \in \mathfrak{B}_n^{(r)}-\mathfrak{D}_n^{(r)}$, if the indices of non-empty components of $\mu$ are not covered by those of $\lambda$, then $[W(\lambda),L(\mu)]=0$. Further, we have the following result:

\begin{cor}
	Let $\lambda_{k_1,k_2},\mu_{l_1,l_2}\in \mathfrak{B}_n^{(r)}-\mathfrak{D}_n^{(r)}$ be two one-column multipartitions of $n$ such that the $k_1$ and $k_2$ components in $\lambda_{k_1,k_2}$ and $l_1$ and $l_2$ components in $\mu_{l_1,l_2}$ are non-empty. If $\{k_1,k_2\} \cap \{l_1,l_2\} =\emptyset$, then
	\begin{equation*}
		Hom(W(\lambda_{k_1,k_2}),W(\mu_{l_1,l_2}))=0.
	\end{equation*}
\end{cor}
\begin{proof}
	For any $\nu\in\mathfrak{B}_n^{(r)}-\mathfrak{D}_n^{(r)}$, as $\{k_1,k_2\} \cap \{l_1,l_2\} =\emptyset$, the theorem above implies 
	\begin{equation}\label{eq96}
		[W(\lambda_{k_1,k_2}),L(\nu)]\times[W(\mu_{l_1,l_2}),L(\nu)]=0.
	\end{equation}
    If $Hom(W(\lambda_{k_1,k_2}),W(\mu_{l_1,l_2}))\neq 0$, then $Hom(W(\lambda_{k_1,k_2}),L(\nu))\neq 0$ for some $\nu$ such that $[W(\mu_{l_1,l_2}),L(\nu)]\neq 0$. On the other hand, $Hom(W(\lambda_{k_1,k_2}),L(\nu))\neq 0$ implies that $[W(\lambda_{k_1,k_2}),L(\nu)]\neq 0$, which is contradictory to $(\ref{eq96})$.
\end{proof}

\section{Further problems}

One of the most important properties of the Temperley-Lieb algebras is that they may be described in terms of planar diagrams.
In \cite{Kauffman1990AnIO} Kauffman shows that the Temperley-Lieb algebra $TL_n(q)$ has a basis consisting of planar $(n,n)$-diagrams. In \cite{Martin_1994}, Martin and Saleur define the Temperley-Lieb algebra of type $B_n$ as an associative algebra with blob $(n,n)$-diagrams as a basis.

It is still an open question how to find a diagrammatic description of our generalised Temperley-Lieb algebra $TL_{r,1,n}$. Graham and Lehrer reveal the connection between cellular bases and diagrammatic bases of Temperley-Lieb algebras of type $B_n$ and affine Temperley-Lieb algebras in \cite{GRAHAM2003479}.
Theorem \ref{cb33} gives us a cellular basis of our $TL_{r,1,n}$ with respect to the KLR generators. Although there is natural grading with this basis, a direct connection to blob diagrams is still not clear.

Another question is whether there is a further generalisation of the Temperley-Lieb algebras corresponding to all unitary reflection groups. We have made a significant progress in this direction. For any imprimitive complex reflection group $G(r,p,n)$, the corresponding Temperley-Lieb algebra $TL_{r,p,n}$ can be defined in a similar way.  This work will be described in an upcoming work.


\bibliography{refs1}
\bibliographystyle{siam}
\end{document}